\documentclass[a4paper,12pt]{amsart}

\usepackage{amssymb,amsbsy,amsmath,amsfonts,amssymb,amscd,graphicx,ifpdf,hyperref}
\usepackage{latexsym}
\usepackage{geometry}

\input xy
\xyoption{all}

\ifpdf
\DeclareGraphicsExtensions{.pdf,.png,.jpg,.mps}
\else
\DeclareGraphicsExtensions{.eps}
\fi

\newcommand\A{{\mathcal A}}

\newcommand{\CC}{\ensuremath{\mathbb{C}}}

\newcommand{\ZZ}{\ensuremath{\mathbb{Z}}}
\newcommand{\QQ}{\ensuremath{\mathbb{Q}}}

\newcommand{\NN}{\ensuremath{\mathbb{N}}}
\newcommand{\hol}{\ensuremath{\mathcal{O}}}

\newcommand{\PP}{\ensuremath{\mathbb{P}}}

\def\eea{\end{eqnarray*}}
\def\bea{\begin{eqnarray*}}

\newcommand\dual{\mathrel{\raise3pt\hbox{$\underline{\mathrm{\thinspace d
\thinspace}}$}}}
\newcommand\qe{\ifhmode\unskip\nobreak\fi\quad $\Box$}       % box for QED

\def\BOX{\hfill\lower.5\baselineskip\hbox{$\Box$}}
\newtheorem{theo}{Theorem}[section]
\newtheorem{remarkk}[theo]{Remark}
\newenvironment{rem}{\begin{remarkk}\rm}{\end{remarkk}}

\newtheorem{defin}[theo]{Definition}

\newtheorem{prop}[theo] {Proposition}
\newtheorem{cor}[theo]{Corollary}
\newtheorem{lemma}[theo]{Lemma}
\newtheorem{example}[theo]{Example}

\newtheorem{conj}[theo]{Conjecture}
\newtheorem{problem}[theo]{Problem}

\DeclareMathOperator{\coker}{coker}

\DeclareMathOperator{\Sing}{Sing}
\DeclareMathOperator{\lcm}{lcm}

\begin{document}

\title{Product-Quotient Surfaces: new invariants and algorithms}
\author{I. Bauer,  R. Pignatelli}

\thanks{This research took part in the realm of the
Forschergruppe  790 `Classification of algebraic surfaces and compact complex manifolds' of the D.F.G. The second author was partially supported by the projects PRIN2010-2011 Geometria delle variet\`a algebriche and Futuro in Ricerca 2012 Moduli Spaces and Applications.}
\date{\today}
\maketitle
\begin{abstract}
In this article we suggest a new approach to the systematic, computer-aided construction and to the classification of product-quotient surfaces, introducing a new invariant, the integer $\gamma$,
which depends only on the singularities of the quotient model $X = (C_1 \times C_2)/G$. It turns out that $\gamma$  is related to the codimension of the subspace of $H^{1,1}$
generated by algebraic curves coming from the construction (i.e., the classes of the two fibers and the Hirzebruch-Jung strings arising from the minimal resolution of singularities of $X$).
 
Profiting from this new insight we developped and implemented an algorithm which constructs all regular product-quotient surfaces with given 
values of $\gamma$ and  geometric genus in  the computer algebra program MAGMA. Being far better than the previous algorithms,  we are able to construct  a substantial number of new 
regular product-quotient surfaces of geometric genus zero. We prove that only two of these are of general type, raising the number of known families of product-quotient 
surfaces of general type with genus zero to 75. This gives evidence to the conjecture that there is an effective bound $\Gamma(p_g,q) \geq \gamma$ (cf. Conjecture \ref{boundc}). 
\end{abstract}

\section{Introduction}
Let $G$ be a finite group acting on two compact Riemann surfaces $C_1$, $C_2$ of respective genera $g_1, g_2 \geq 2$.
We shall consider the diagonal action of $G$ on $C_1 \times C_2$ and
in this situation we say for short: the action of $G$ on
$C_1 \times C_2$ is {\em unmixed}. By \cite{FabIso} we may assume w.l.o.g.
that $G$ acts faithfully on both factors.
\begin{defin}\label{prodquot}
The minimal resolution $S$ of the singularities of $X = (C_1 \times C_2)/ G$, where $G$ is a finite group with an unmixed
action on the direct product of two compact Riemann surfaces $C_1$, $C_2$ of respective genera at least
two, is called a {\em product-quotient surface}.

$X$ is called  {\em the quotient model} of the product-quotient surface.
\end{defin}

In the last years several people have been studying product-quotient surfaces and quite some literature is nowadays available (cf. {\it e.g} \cite{FabIso, bacat, pg=q=0, polmi, PolizziNumProp, frap, 4names, bp, frapi, pene, gape}...).

The authors (partially in collaboration with F. Catanese, D. Frapporti and F. Grunewald) have been focussing mainly on  the  systematic construction and classification of product-quotient surfaces of general type with geometric genus $p_g=0$. 
Our previous results may be summarized as follows.

\begin{theo}[\cite{bacat}, \cite{pg=q=0} \cite{4names},\cite{bp}]\label{classiso} \

1) Product-quotient surfaces {\em isogenous to a product} ({\it i.e.} $G$ acts freely)
with $p_g(S) = q(S) = 0$ form 13 irreducible connected components of the Gieseker 
moduli space of surfaces of general type.

2) {\em Minimal} product-quotient surfaces with $p_g=0$  of general type form 72 irreducible families, including the 13 in point 1.

3) There is exactly one product-quotient surface with $p_g=0$, $K^2_S >0$ which is not minimal.
\end{theo}

Even if quite some effort has been put and new techniques have been developped, it remains the following big open problem:
\begin{problem}\label{problem}
 Classify all product-quotient surfaces of general type with $p_g=0$.
\end{problem}
By theorem \ref{classiso} it remains to classify all {\em non-minimal} product-quotient surfaces of general type with geometric genus zero.
In \cite{bp} we wrote a MAGMA script producing 
all regular  product-quotient surfaces with $p_g=0$ and 
fixed $K^2_S$. The problem is that this program is very slow for $K^2_S <0$.

As already 
noticed in {\it loc. cit}, one approach to solve the above question is
\begin{itemize}
\item[1)]  prove that $K_S^2 \leq -C$ implies that $S$ is not of general type for some explicit integer $C$;
\item[2)]  use a suitable algorithm to produce all regular product-quotient surfaces with $p_g=0$ and $-C<K^2<0$.
\end{itemize}
At the moment, not only an explicit bound is out of reach, but also the algorithm used in \cite{bp} is far from being good enough to make step 2 work even for small $C$.

In the present article we suggest a different approach to solve problem \ref{problem}. 

The key observation is the following: inspecting the list of  surfaces in Theorem \ref{classiso} (cf. \cite{bp}, tables 1, 2), one notices that all minimal product-quotient surfaces with $p_g=0$ have the property that $H^{1,1}(S)$ is generated by the fibres of the two fibrations and the irreducible components of the 
exceptional divisor of the minimal resolution of singularities $\sigma$, whereas for the single non-minimal  product-quotient surface with $K_S^2 >0$, this is not the case. Here the fibres and the exceptional curves generate a subspace of codimension 2. 

This remark led us to study, for a general product-quotient surface $S$, the subspace of $H^{1,1}(S)$ generated by the fibres of the two fibrations and the irreducible components of the 
exceptional divisor of the minimal resolution of singularities $\sigma$. We shall  prove here, that its codimension is even, and equals 
$2(p_g(S) + \gamma)$ (cf. Proposition \ref{gamma+pg}), where $\gamma$ is an invariant depending only on some numerical data of the singularities of $X$. 

Note that then in particular: $p_g=0 \Rightarrow \gamma \geq 0$. 

\begin{rem}
Looking at the program used in \cite{bp}, one notices that almost half of the computations 
when running for $p_g=0$ had to deal with the case $\gamma <0$. This information could of course now be used to speed up the computations quite a bit.

Instead, we chose to write a different program, substituting (as input) $\gamma$ to $K^2$. 
The result is a much quicker program, producing dozens of new regular product-quotient surfaces 
with $p_g=0$ (and several with $p_g>0$, on which we do not report here). 
\end{rem}

The computations we did suggest the following
\begin{conj}
Let $S$ be a product-quotient surface. Then $S$ is minimal if and only if $p_g(S) + \gamma =0$.
\end{conj}
We shall prove the conjecture for surfaces with vanishing geometric genus (cf. Theorem \ref{gamma0}).

Running our program  for $\gamma=1,2,3$, produces three examples of surfaces of general type, two with $\gamma=1$ 
(including the surface in Theorem \ref{classiso}, 4), and one with $\gamma=2$: the two new examples, 
 both  Numerical Godeaux surfaces, are described
in Section \ref{constructions}. Together with the results \cite{bp} we have 75 families of product-quotient surfaces of general type with $p_g=q=0$ and we conjecture that these are all.

What we can prove, is the following:
\begin{prop}
Let $S$ be a product-quotient surface of general type with $p_g=0$ not among the $75$ families just mentioned.
Then

\begin{itemize}
\item either $\gamma \geq 4$,
\item or $\gamma=3$ and $X$ has a singular point of multiplicity at least $14$,
\item or $\gamma=2$ and $X$ has a singular point of multiplicity at least $45$.

\end{itemize}
\end{prop}
On the way to prove the above we produce a substantial number of product-quotient surffaces not of general type, collected in the 
tables \ref{K2<1},  \ref{gamma1}, \ref{gamma2}, \ref{gamma2b} and \ref{gamma3}.

Coming back to Problem \ref{problem}, this new approach allows to substitute  part 1) of the proposed solution to the classification problem of regular product-quotient surfaces with $p_g=0$ by
 Conjecture \ref{boundc}. 
 \begin{conj}
 There is an explicit function $\Gamma = \Gamma(p_g,q)$ such that, for the quotient model $X$ of every 
product-quotient surface $S$ of general type
$$
\gamma(X) \leq \Gamma(p_g(S),q(S)).
$$
 \end{conj}
 We give some motivation to this Conjecture in Section \ref{evidences}, proving the above conjecture under some additional hypothesis.
 
  Finally in section 
\ref{duality} we construct a duality among regular product-quotient surfaces allowing, among other things, 
to give a new interpretation of the "half-codimension" $p_g+\gamma$, which in fact turns out to be equal to the geometric genus of the dual 
product-quotient surface.

\section{Notation}

This chapter is dedicated to fix once and for all the notation, which should be valid througout the paper.

Let $C$ be an algebraic curve, $G$ a finite group acting faithfully on it, 
$C'=C/G$. We associate to the pair $(C,G)$, after certain choices on 
$C/G$ (\cite[Section 4]{surveypg=0} for details), an 
\begin{itemize}
\item {\it appropriate orbifold homomorphism} 
$\varphi \colon {\mathbb T}(g(C/G); m_1, 
\ldots, m_r) \rightarrow G$,
\end{itemize}
 which allows (up to the above made choices) to reconstruct $(C,G)$.
  
Equivalently, one can give 
\begin{itemize}
\item a {\it generating vector} (\cite[Definition 1.1]{PolizziNumProp}) 
of $G$ of {\it signature} (or {\it type}) $(g(C/G); m_1, \ldots, m_r)$, 
\end{itemize}
where 
$g(C/G)$ is the genus of the quotient curve. 

We will say that the action of $G$ on $C$ has {\em signature} $(g(C/G); m_1, \ldots, m_r)$.

We will also need the number 
\begin{itemize}
\item $
\Theta:=\Theta(g(C/G); m_1, \ldots, m_r) := 2(g(C/G))-2+\sum \left( 1 - \frac{1}{m_i} \right)>0
$,
\end{itemize}
which relates the genus of $C$ and the order of $G$ by the Hurwitz formula
\begin{equation}\label{Hurwitz}
2g(C)-2=|G|\Theta
\end{equation}

In the following $C_1$, $C_2$ will be two algebraic curves of respective genera $g_1,g_2 \geq 2$, $G$ a finite group
acting faithfully on both curves. 

We consider the quotient surface $X:=(C_1 \times C_2)/G$ by the diagonal action, and the minimal 
resolution of its singularities $\sigma \colon S \rightarrow X$. We will refer to $S$ as 
\begin{itemize}
\item a 
{\it product-quotient} surface and 
\item to $X$ as its {\it quotient model}. 
\end{itemize}
We will denote by $\overline{S}$ 
the minimal model of $S$.

As usual, $p_g(S)$ (or simply $p_g$) will be the geometric genus $h^2(\hol_S)$, and 
$q(S)$ (or simply $q$) will be the irregularity $h^1(\hol_S)$. We will also denote by $\chi$ or $\chi(S) = 1-q+p_g$ the Euler 
characteristic of the structure sheaf of $S$.

 We will say that the quotient model $X$ has {\em type}
$$
\left((g(C_1/G); m_1, \ldots, m_r), (g(C_2/G); n_1, \ldots, n_s) \right),
$$
if the action of $G$ on $C_1$ has signature $(g(C_1/G); m_1, \ldots, m_r)$ and 
the action of $G$ on $C_2$ has signature $(g(C_2/G); n_1, \ldots, n_s)$; we will write 
$\Theta_1$ for $\Theta(g(C_1/G);$ $m_1, \ldots, m_r)$ and $\Theta_2$ for $\Theta(g(C_2/G); 
n_1, \ldots, n_s)$.

All singularities of $X$ are cyclic quotient singularities, 
locally isomorphic to the quotient of $\CC^2$ by the cyclic group generated by $(x,y) \mapsto 
(e^\frac{2\pi i}n x,e^\frac{2q\pi i}n y)$ for two relatively prime positive integers 
$q,n$ with $q<n$. 
We will say that the singularity is of type $\frac{q}{n}$, instead of using the classical notation 
$\frac{1}{n}(1,q)$. 

We denote by  $q'$ the integer between  $1$ and $n-1$ which is the multiplicative inverse of $q$ 
modulo $n$, whence a singularity of type  $\frac{q}{n}$ is also of type  $\frac{q_1}{n_1}$ if and only 
if $n=n_1$ and $q_1$ is either $q$ or $q'$.

We associate four numbers to each cyclic quotient singularity, depending only on its type.
\begin{defin}\label{numbers}
For each rational number $0<\frac{q}{n}<1$ we consider its continued fraction
$$
\frac{n}{q} = b_1 - \frac{1}{b_2 - \frac{1}{b_3 - \ldots}} =:[b_1, \ldots, b_l];
$$
writing $\frac{q}{n}=[b_1,\ldots,b_l]$, $b_i \in \NN$, $b_i \geq 2$;
then we define $l\left(\frac{q}{n}\right)$, $\gamma\left(\frac{q}{n}\right)$, 
$\mu\left(\frac{q}{n}\right)$, $I\left(\frac{q}{n}\right)$ as follows.
\begin{itemize}
\item $l\left(\frac{q}{n}\right)$ is the length of the continued fraction;
\item $\gamma\left(\frac{q}{n}\right):=\frac16 \left[ \frac{q+q'}n + \sum_{i=1}^l (b_i-3)\right]$;
\item $\mu\left(\frac{q}{n}\right)=1-\frac1n$.
\item $I\left(\frac{q}{n}\right)=\frac{n}{\gcd(n,q+1)}$.
\end{itemize}
\end{defin}
It is well known that if $\frac{q}{n}=[b_1,\ldots,b_l]$, then  $\frac{q'}{n}=[b_l,\ldots,b_1]$.
It follows immediately  that $l,\gamma$, $\mu$ and $I$ do not change when substituting $q$ with $q'$, and therefore the following definition is well posed.
\begin{defin}
Let $x$ be a singular point of $X$, of type $\frac{q}{n}$.
Then we define $l_x:=l\left(\frac{q}{n}\right)$; $\gamma_x:=\gamma\left(\frac{q}{n}\right)$;
$\mu_x:=\mu\left(\frac{q}{n}\right)$; $I_x:=I\left(\frac{q}{n}\right)$.
\end{defin}

The {\em basket of singularities} $\mathfrak B$ of the quotient model $X$ is the {\it multiset} 
(which means that it is a set and moreover a positive 
integer is associated to each of its elements, its {\it multiplicity}) of rational numbers. For example, a surface with two nodes has basket $\{2 \times \frac12\}$.

We globalize $l$, $\gamma$, $\mu$ and $I$ as follows.

\begin{defin}[\bf Invariants of the basket $\mathfrak{B}$]
Let ${\mathfrak B}$ be the basket of singularities of the quotient model $X$ of a product-quotient surface $S$. Then
$$
l(X):=\sum_{x \in {\mathfrak B}} l_x;\ \
\gamma(X):=\sum_{x \in {\mathfrak B}} \gamma_x;\ \
\mu(X) :=\sum_{x \in {\mathfrak B}} \mu_x;\ \
I(X) :=\lcm_{x \in {\mathfrak B}} I_x.
$$
\end{defin}
\begin{rem}
$I$ is the index of $X$, the minimal positive integer such that $IK_X$ is Cartier.
It is  the only number, among the numbers defined in Definition \ref{numbers}, which was already considered in \cite{bp}. The numbers $l$ ,$\gamma$ and $\mu$ are convenient substitutes of the numbers $e$, $k$, and $B$ considered in \cite{bp}.
For the convenience of the reader, we recall the definition of $e$, $k$ and $B$ in terms of them:
\begin{equation}\label{ekB=lmg}
e=l+\mu;\ \ \
k=6\gamma+l-2\mu; \ \ \
B=3\left(2\gamma+l\right).
\end{equation}
\end{rem}

\section{Hodge theory of Product-Quotient Surfaces}
We start with the following:

\begin{prop}\label{h2}\ 
 \begin{enumerate}
\item $\forall k \neq 2$, $H^k(S, \CC) \cong H^k (X, \CC)$.
 \item $H^2(S, \CC) \cong H^2 (X, \CC) \oplus \CC^{l(X)}$
 \end{enumerate}
\end{prop}
\begin{proof}
1) Let $X^{\circ}$ be the smooth locus of $X$. For each singular point $x$ of $X$, choose a small neighbourhood 
$U_x$ of $x$ which may be retracted to the point $x$, set 
\begin{itemize}
\item $U:=\bigcup U_x$ and define 
\item $U_x^\circ:=U_x \setminus \{x\} = U_x \cap X^\circ$, 
\item $U^\circ=U \cap X^\circ$. 
\end{itemize}
We also consider 
\begin{itemize}
\item $S^\circ:=\sigma^{-1}(X^\circ)$, 
\item $V_x:=\sigma^{-1}(U_x)$,  $V_x^\circ:=\sigma^{-1}(U_x^\circ)$, 
\item $V:=\sigma^{-1}(U)$, $V^\circ:=\sigma^{-1}(U^\circ)$.
\end{itemize}
The Mayer-Vietoris exact sequences  corresponding to the decompositions 
$$X=X^\circ \cup U, \ \ S=S^\circ \cup V$$
give a commutative diagram
\begin{small}
\begin{equation}\label{mv}
\xymatrix{
H^{k-1}(X^\circ) \oplus H^{k-1}(U) \ar[r] \ar^{b_{k-1} \oplus c_{k-1}}[d]&H^{k-1}(U^\circ) \ar[r] \ar^{d_{k-1}}[d] &H^k(X) \ar[r] \ar^{a_k}[d] &H^k(X^\circ) \oplus H^k(U) \ar[r] \ar^{b_k \oplus c_k}[d] & H^k(U^\circ)\ar^{d_k}[d]\\
H^{k-1}(S^\circ) \oplus H^{k-1}(V)\ar[r]&H^{k-1}(V^\circ) \ar[r]&H^k(S) \ar[r]&H^k(S^\circ) \oplus H^k(V) \ar[r]&H^k(V^\circ)\\
}
\end{equation}
\end{small}
whose vertical maps are the maps induced  in cohomology by the suitable restrictions of $\sigma$.

Since $\sigma_{|S^\circ}$ and $\sigma_{|V^\circ}$ are homeomorphisms, all the maps $b_q$ and $d_q$ are isomorphisms. Moreover, since $U_x$ retracts to a point and $V_x$ to a tree of $l_x$ rational curves, $c_k$ is an isomorphism for all $k \neq 2$, and the (injective) map $0 \rightarrow \CC^l$ for $k=2$.

By the Five Lemma, it follows that all maps $a_k$ with $k \neq 2,3$ 
are isomorphisms, while the Four Lemmas imply that $a_2$ is injective and $a_3$ is surjective.

Let $A_1, \ldots, A_l$ be the exceptional divisors of $\sigma$. 
Since $V$ retracts to the union of the $A_i$, the inclusions yield an isomorphism 
$H^2(V) \cong \oplus_1^l H^2(A_i)$, so $H^2(V)\cong \CC^l$. Moreover, identifying by Poincar\'e 
duality $H^2(S)$ with $H_2(S)^*$, the map $H^2(S) \rightarrow H^2(A_i) \cong \CC$ induced by the inclusion 
sends each operator $\phi$ onto $\phi(A_i)$. Since the intersection form on the $A_i$ is negative 
definite, it follows that the map $H^2(S) \rightarrow H^2(V) \cong \oplus H^2(A_i)$ is surjective.

Then a standard diagram chasing argument on (\ref{mv}) shows that $a_3$ is injective, so an 
isomorphism.

2) We have just shown that all maps $a_k$,$b_k$, $c_k$ and $d_k$ are isomorphisms, except $a_2$ and 
$c_2$. Moreover, $a_2$ and $c_2$ are injective, and $\dim (\coker c_2)=l$. Since the alternating sum of the dimensions of the vector 
spaces in a finite exact sequence is zero, comparing the two long exact sequences in (\ref{mv}) we obtain 
$\dim H^2(S) = \dim H^2(X) + l$. 
\end{proof}

We shall focus now on $H^2(X,\CC)$.

\begin{prop}\label{parity}\ 
\begin{itemize}
 \item $\dim H^2(X,\CC) \equiv 0 \mod 2$,
\item $\dim H^2(X,\CC)) \geq 2$.
\end{itemize}

\end{prop}
\begin{proof}
By the Hodge decomposition we know that
\begin{align*}
H^2(C_1 \times C_2,\CC) \cong& H^0(\Omega^2_{C_1 \times C_2}) \oplus H^1(\Omega^1_{C_1 \times C_2}) 
\oplus H^2(\hol_{C_1 \times C_2})\\
\cong & H^0(\Omega^2_{C_1 \times C_2}) \oplus H^1(\Omega^1_{C_1 \times C_2}) 
\oplus H^0(\Omega^2_{C_1 \times C_2})^*.
\end{align*}
Therefore the $G$-invariant part of $H^2(C_1 \times C_2,\CC)$ decomposes as 
\begin{align*}
H^2(X,\CC)) \cong & H^2({C_1 \times C_2},\CC)^G \\
\cong & H^0(\Omega^2_{C_1 \times C_2})^G \oplus H^1(\Omega^1_{C_1 \times C_2})^G
\oplus (H^0(\Omega^2_{C_1 \times C_2})^*)^G \\
\cong & H^0(\Omega^2_{C_1 \times C_2})^G \oplus H^1(\Omega^1_{C_1 \times C_2})^G
\oplus (H^0(\Omega^2_{C_1 \times C_2})^G)^*.
\end{align*}

Therefore, writing as usual $h^q$ for the dimension of $H^q$, $h^2(X,\CC) = 2 \cdot h^0(\Omega^2_{C_1\times C_2})^G + h^1(\Omega^1_{C_1\times C_2})^G$, whence the claim is proven once we show that $h^1(\Omega^1_{C_1\times C_2})^G \equiv 0 \mod 2$.

We recall that 
by K\"unneth's formula (cf. e.g. \cite{kaup}) and Hodge theory

\begin{align*}
H^1(\Omega^1_{C_1 \times C_2})\cong& \left( H^1(\Omega^1_{C_1}) \otimes H^0(\hol_{C_2}) \right) \oplus
\left( H^1(\Omega^1_{C_2}) \otimes H^0(\hol_{C_1}) \right)\\
\oplus& \left( H^0(\Omega^1_{C_1}) \otimes H^1(\hol_{C_2}) \right) \oplus
\left( H^0(\Omega^1_{C_2}) \otimes H^1(\hol_{C_1}) \right)\\
\cong& \left( H^1(\Omega^1_{C_1}) \otimes H^0(\hol_{C_2}) \right) \oplus
\left( H^1(\Omega^1_{C_2}) \otimes H^0(\hol_{C_1}) \right)\\
\oplus& \left( H^0(\Omega^1_{C_1}) \otimes \overline{H^0(\Omega^1_{C_2})} \right) \oplus \left( \overline{H^0(\Omega^1_{C_1})} \otimes H^0(\Omega^1_{C_2}) \right).\\
\end{align*}

It is wellknown that if $\chi$ is the character of the $G$-module  $H^0(\Omega^1_{C_1})$, then $\bar{\chi}$ is the character of the  the $G$-module  $\overline{H^0(\Omega^1_{C_1})}$. From this fact it follows that
$$
\left( H^0(\Omega^1_{C_1}) \otimes \overline{H^0(\Omega^1_{C_2})} \right)^G \oplus \left( \overline{H^0(\Omega^1_{C_1})} \otimes H^0(\Omega^1_{C_2}) \right)^G \cong V \oplus \bar{V},
$$
where

$$
V:= \left( H^0(\Omega^1_{C_1}) \otimes \overline{H^0(\Omega^1_{C_2})} \right)^G.
$$

Since the fundamental class of $C_i$ is $G$-invariant, we have
$$
\left( H^1(\Omega^1_{C_1}) \otimes H^0(\hol_{C_2}) \right) \oplus
\left( H^1(\Omega^1_{C_2}) \otimes H^0(\hol_{C_1}) \right) =
$$

$$
\left( H^1(\Omega^1_{C_1}) \otimes H^0(\hol_{C_2}) \right)^G \oplus
\left( H^1(\Omega^1_{C_2}) \otimes H^0(\hol_{C_1}) \right)^G \cong \CC^2.
$$

This proves the claim.
\end{proof}

Consider the inclusion 
$$
j \colon U:= X \setminus \Sing(X) \rightarrow X
$$
and define $\tilde{\Omega}^p_X:= j_* \Omega^p_U$. 

\begin{theo}[\cite{steenbrink}, (1.10),(1.11), (1.12)]\label{steenbrink}
\begin{enumerate}
 \item $\tilde{\Omega}^p_X$ is coherent for all $p$;
\item $\tilde{\Omega}^p_X = \sigma_* \Omega^p_S$, $\forall$ p;
\item $\tilde{\Omega}^p_X = (\pi_* \Omega^p_{C_1 \times C_2})^G$;
\item there is a morphism of spectral sequences 
\begin{equation*}\label{boh}
\xymatrix{
E^{pq}_1 =H^q(X,\tilde{\Omega}^p_X) \ar^{\sigma^*}[d] & \implies & H^{p+q}(X,\CC)\ar^{\sigma^*}[d]\\
E'^{pq}_1 =H^q(S,\Omega^p_S)  & \implies & H^{p+q}(S,\CC).\\ 
}
\end{equation*}
which is injective on $E_1$-level.
\end{enumerate}
\end{theo}

\begin{prop}\label{hodge}\ 
 If $p_g(S) = 0$, then $H^0(C_1 \times C_2, \Omega^2_{C_1 \times C_2})^G = 0$. In particular, $H^2(X, \CC) \cong H^1(C_1 \times C_2, \Omega^1_{C_1 \times C_2})^G$.
\end{prop}
\begin{proof}
By Theorem \ref{steenbrink}, $H^0(X,\tilde{\Omega}^2_X) \rightarrow H^0(S,\Omega^2_S) = 0$ is 
injective, and $H^0(X,\tilde{\Omega}^2_X) = H^0(C_1 \times C_2, \Omega_{C_1 \times C_2}^2)^G$.
\end{proof}

We recall the following version of Schur's lemma (cf. {\it e.g.} \cite[Proposition 4]{langrep}):

\begin{lemma} \label{schur}
 Let $G$ be a finite group and let $W$ be an irreducible $G$-representation. Then
\begin{enumerate}
 \item $\dim (W \otimes W^*)^G = 1$;
\item if $W'$ is an irreducible $G$-representation not isomorphic to $W^*$, then $\dim (W \otimes W')^G = 0$.
\end{enumerate}
\end{lemma}

\begin{rem} \
\begin{enumerate}
 \item Proposition \ref{hodge} shows that the singularities of the quotient-model $X$ give no conditions of adjunction for canonical forms, even if the singularities are not canonical. This is not true for the bicanonical divisor.
\item The above results (especially the proof of prop. \ref{parity}) make clear that the condition that $S$ has vanishing geometric genus gives strong restrictions on the $G$-modules $H^0(C_i, \Omega^1_{C_i})$. For example, using Schur's lemma, we can list the following properties:
\begin{enumerate}
 \item if $\chi$ is an irreducible character of $G$, then $H^0(\Omega^1_{C_1})^{\chi} = 0$ or $H^0(\Omega^1_{C_2})^{\bar{\chi}} = 0$;
\item $\dim H^2(X, \CC) > 2$ if and only if there is an irreducible non selfdual  character  $\chi$ of $G$ such that $H^0(\Omega^1_{C_1})^{\chi} \neq 0$ and $H^0(\Omega^1_{C_2})^{\chi} \neq 0$.
\end{enumerate}
Each time that such a situation occurs, it raises the dimension of $\dim H^2(X, \CC)$ by two. 
\end{enumerate}
\end{rem}
An immediate consequence of the above considerations is the following:

\begin{prop}\label{Gselfdual}
Let $S$ be a regular product-quotient surface with $p_g=0$,  quotient model $X = (C_1 \times C_2)/G$, such that all irreducible representations of $G$ are selfdual (e.g. $G=\mathfrak{S}_n$). Then $h^2(X,\CC) = 2$.
\end{prop}
\section{The invariant $\gamma$}

The formulas for $K^2_S$, $\chi$ and $q$ in \cite{bp} translate, in the notation of the present paper, as follows.

\begin{prop}[\cite{bp}, Prop. 1.6 and Cor. 1.7 and \cite{serrano}]\label{forminf}
$$K^2_S=8\chi-2\gamma-l,\ \ \ \ \
\chi=\frac{(g_1-1)(g_2-1)}{|G|}+\frac{\mu-2\gamma}{4},\ \ \ \ \
q=g_1+g_2.$$
\end{prop}

Observe that the new invariant  $\gamma$ is (as defined in \ref{numbers}) a priori a rational number. 

But, in fact, we are going to show in the
next  proposition  that $\gamma$ is an integer, bounded from below by $-p_g(S)$.
\begin{prop}\label{gamma+pg}
$$
\gamma(X)+p_g(S) \in \NN.
$$
Moreover, if $\gamma(X)+p_g(S)=0$, then $S$ has {\em maximal Picard number.}
\end{prop}
\begin{proof}
The intersection form on $H^2(S, \CC)$ shows that the fibres of the two fibrations $S \rightarrow C_i/G$, and the $l$ irreducible exceptional curves of $\sigma$ form a set of $l+2$ linearly independent classes in $H^1(S, \Omega^1_S)$. Therefore we have 
$$
h^{1,1}-l-2\in \NN.
$$
By Proposition \ref{h2},we know that $\dim H^2(S, \CC)=l+\dim H^2(X, \CC)$ and, by Proposition \ref{parity}, we see that $h^{1,1}$ has the same parity as $l$. Therefore $h^{1,1}-l-2\in 2\NN$.

The claim follows, using Noether formula and Hodge theory, since
\begin{align*}
2(\gamma+p_g)=&-K^2_S+8\chi-l+2p_g\\
=&c_2(S)-4\chi-l+2p_g\\
=&2-2b_1+b_2-4+4q-4p_g-l+2p_g\\
=&h^{1,1}-l-2.\\
\end{align*}
In particular, if $\gamma(X)+p_g(S)=0$, then $H^{1,1}(S)$ is generated by algebraic curves 
(the fibres of the two fibrations and the exceptional curves of $\sigma$)
and therefore has maximal Picard number.
\end{proof}
On the other hand the next proposition implies that the possible values of $\gamma$ 
distribute symmetrically around zero.
\begin{prop}\label{gammasym}
$\gamma(\frac{q}{n})=-\gamma(\frac{n-q}n)$.
\end{prop}
\begin{proof}
Write $\frac{n}q=[b_1,\ldots,b_l]$, $\frac{n}{n-q}=[a_1,\ldots,a_k]$. Then by 
\cite[Lemma 4]{Riemenschneider} 
$$\sum_1^k (a_i-1)=\sum_1^l (b_i-1)=k+l-1.$$
Therefore
\begin{multline*}
6\left(\gamma \left( \frac{q}{n} \right)+\gamma \left( \frac{n-q}n \right) \right)=
\frac{q+q'}n +\frac{n-q+n-q'}n + \sum_{i=1}^l (b_i-3)+ \sum_{i=1}^k (a_i-3)= \\
=2 + \sum_{i=1}^l (b_i-1) -2l + \sum_{i=1}^k (a_i-1) -2k=0.
\end{multline*}
\end{proof}

Having proven a lower bound for $\gamma$ in terms of the invariants of $S$, 
it is then reasonable to imagine that some upper bound should hold too. In other words 
\begin{conj}\label{boundc}
 There is an explicit function $\Gamma = \Gamma(p_g,q)$ such that, for the quotient model $X$ of every 
product-quotient surface $S$ of general type
$$
\gamma(X) \leq \Gamma(p_g(S),q(S)).
$$
\end{conj}

\begin{rem}
From Proposition \ref{h2} and the proof of Proposition \ref{gamma+pg}, $h^2(X,\CC)=2(\gamma+2p_g+1)$. In particular,  by Proposition \ref{Gselfdual},
if $S$ is a regular product-quotient surface with $p_g=0$,  quotient model $X = (C_1 \times C_2)/G$, such that all irreducible representations of $G$ are selfdual, then $\gamma = 0$.
\end{rem}
\section{A classification algorithm for surfaces of general type with given $p_g$, $q$ and $\gamma$}
In \cite{bp} we gave an algorithm producing all product-quotient surfaces with given (as input) 
values of $K_S^2$, and $\chi(\hol_S)$.
 
In the following we shall show that we can substitute $\gamma$ to $K^2_S$; 
in other words, fixing $\chi$ and $\gamma \in \NN$, we also get a finite problem. In particular, 
answering in the affirmative Conjecture \ref{boundc} we would have an algorithm constructing all
product-quotient surfaces with fixed values of $q$ and $p_g$.

To ease the forthcoming formulas, we also introduce the following:
\begin{defin}\label{xi}
$\xi:=\xi(X):=4\chi+2\gamma-\mu\in \QQ$.
\end{defin}
\begin{rem}\label{xipq}
Observe  that $\xi$ only depends on $\chi$ and ${\mathfrak B}$. Moreover,
$$
\xi(X)=\frac{4(g_1-1)(g_2-1)}{|G|}=\frac{K^2_X}2.
$$ 
\end{rem}

We recall a theorem by Xiao Gang:

\begin{theo}[\cite{xiao}]
Let $T$ be a minimal surface of general type and $G$ a finite group of automorphisms of $T$, such that $T/G$ is of general type. Let $Y$ be the minimal model of a resolution of singularities of  $T/G$. Then
 $$
1 \leq K_Y^2 \leq \frac{K^2_T}{|G|}.
$$
\end{theo}
Using remark \ref{xipq} we immediately get the following lower bound for $\xi$.

\begin{cor}\label{4chi+gamma-mu>=1/2}
$$
\xi(X) \geq \frac12 K_{\bar{S}}^2 \geq \frac 12.
$$
\end{cor}

\begin{proof}
This follows immediately, since $\frac{K^2_{C_1\times C_2}}{|G|}=K^2_{X}=2\xi$.
\end{proof}

We consider the two natural fibrations 
$$
f_1 \colon S \rightarrow C_1/G, \ \ f_2 \colon S \rightarrow C_2/G,
$$ and denote the generic fibre of $f_i$ by $F_i$. Observe that $F_1$ is isomorphic to $C_2$ and $F_2$ is isomorphic to $C_1$. 

These fibrations have been studied in detail in \cite{PolizziNumProp}.
If the type of $X$ is $((g_1;m_1,\ldots,m_r),(g_2;n_1,\ldots,n_s))$, then
 $f_1$ has exactly $r$ reducible fibers, all non reduced, of the form:
$$
F_1 \equiv m_iF_1^{(i)}+\sum a_j A_j, \ \ 1\leq i \leq r,
$$
where the
$A_j$'s are contracted by $\sigma$. Similarly the second fibration
 $f_2 \colon S \rightarrow C_2/G$ with general fibre $F_2$ 
isomorphic to $C_1$, has $s$ reducible fibers of the form $n_iF_2^{(i)}+\sum b_j A_j$.

\begin{rem}[cf. \cite{serrano}, Theorem 2.1]\label{ndividesmi}
Each singular point $x$ of $X$ lies on $\sigma(F_1^{(i)})$ for one $i$.
Moreover, if $x$ is of type $\frac{q}{n}$, then $n$ divides $m_i$.
\end{rem}

We will need the following result by F.  Polizzi, 
computing the self intersection $(F_1^{(i)})^2$ from the types of the singularities of $X$ along $\sigma(F_1^{(i)})$.
\begin{prop}[\cite{PolizziNumProp}, Proposition 2.8]\label{PolizziInt}
$$\sum_{x \in Sing X \cap \sigma(F_1^{(i)})} \frac{q}{n}(x)=-(F_1^{(i)})^2\in \NN,$$ 
where $x$ is a singular point of type $\frac{q}{n}(x)$. 

Moreover, if $x \in \sigma(F_1^{(i)}) \cap \sigma(F_2^{(j)})$ and the contribution of $x$ to 
$(F_1^{(i)})^2$ is $\frac{q}{n}$, then its contribution to $(F_2^{(j)})^2$ is $\frac{q'}{n}$.
\end{prop}

We shall show now (Proposition \ref{bound}) that, fixed $\gamma$, $p_g$ and $q$, we can produce a finite list containing all possible signatures involved in the construction of product quotient surfaces with those values of $\gamma$, $p_g$ and $q$. 

Before doing this, we need to recall further invariants, the integers $\alpha_i$, which were already considered in our previous papers.

\begin{defin}\label{at}
$$
\alpha_1:=\frac{4\chi+2\gamma-\mu}{2\Theta_1}=\frac{\xi}{2\Theta_1},\ \ \ \ \
\alpha_2:=\frac{4\chi+2\gamma-\mu}{2\Theta_2}=\frac{\xi}{2\Theta_2}.
$$
\end{defin}

In fact, we have (cf. e.g. \cite{4names})
\begin{prop}\label{galfa}
$\alpha_i=g_{i+1}-1 \in \NN$.
\end{prop}

\begin{proof}
W.l.o.g. we can assume $i=1$. Then
$$\alpha_1=\frac{\xi}{2\Theta_1}=\frac{2(g_1-1)(g_2-1)}{|G|\Theta_1}=g_2-1\in \NN.$$
\end{proof}

The following inequality allows to bound the multiplicities in the signatures in terms of the genera of the involved curves.
\begin{theo}[\cite{wiman}]\label{4g+2}
Let $H$ be a cyclic group of automorphisms of a compact Riemann surface $C$ of genus $g \geq 2$. Then  $|H| \leq 4g+2$.
\end{theo}

In fact, an immediate consequence of Wiman's inequality is the following:
\begin{cor}\label{cornak}
$$\forall i,j\ \ m_i, n_j \leq  2\min\left( \left(\frac{\xi}{\Theta_1}+3\right),\left(\frac{\xi}{\Theta_2}+3\right) \right),$$
\end{cor}

The next proposition gives upper bounds for $r,s$, $m_i$ and $n_j$, but in terms of $\xi$ and $g(C_i/G)$.
\begin{prop}\label{bound}
The following inequalities hold:
\begin{itemize}
\item[a)] $r\leq \xi+4-2g(C_1/G)$;
\item[b)] if $g(C_1/G)>0$ or $r >3$, then $\forall i$ 
\begin{align*}
m_i \leq&3+ \frac{2\xi+1+\sqrt{(3(4g(C_1/G)+r-3)+2\xi+1)^2-12(4g(C_1/G)+r-3)}}{4g(C_1/G)+r-3}\\
<&6+\frac{4\xi+2}{4g(C_1/G)+r-3}
\end{align*}
\item[c)] if $g(C_1/G)=0$ and $r \leq 3$, then $r=3$ and 
$$m_i \leq 6[\xi+1+\sqrt{\xi(\xi+2)}]< 12(\xi+1)
$$
\end{itemize}
Analogous bounds hold for $s, n_j$.
\end{prop}
\begin{proof}
a) By $1\leq \alpha_1=\frac{\xi}{2\Theta_1}$, $2\Theta_1 \leq \xi$. Since by definition $\Theta_1 \geq 2g(C_1/G)-2 +\frac{r}{2}$,
$$
r \leq 2\Theta_1+4-2g(C_1/G) \leq \xi+4-2g(C_1/G)
$$

b) If $r=0$ there is nothing to prove, so we may assume $r\geq 1$.
Let $m_1$ be the maximum of the $m_i$. Note that  by definition
$$
\Theta_1 \geq 2g(C_1/G)+ \frac{r-3}{2}-\frac1{m_1}=\frac{m_1(4g(C_1/G)+r-3)-2}{2m_1}.
$$
By assumption $m_1(4g(C_1/G)+r-3)-2 \geq 0$. Moreover  $m_1(4g(C_i)+r-3)-2 = 0$ implies that the signature is $(0;2,2,2)$, which implies $\Theta_1=-\frac12<0$, a contradiction.
So $m_1(4g(C_1/G)+r-3)-2 >0$, whence, from corollary \ref{cornak},
$$
m_1 \leq 2\left(\frac{\xi}{\Theta_1} +3\right)\leq 
2\left(\frac{2m_1\xi}{m_1(4g(C_1/G)+r-3)-2} +3\right),
$$
so
$$
m_1^2(4g(C_1/G)+r-3)-2m_1(3(4g(C_1/G)+r-3)+2\xi+1)+12\leq 0.
$$
This immediately implies the desired inequality.

c) By corollary \ref{4chi+gamma-mu>=1/2} we have $\xi \geq \frac 12>0$, and therefore the upper bound for $m_i$ we want to prove is $>6$. Whence we can assume w.l.o.g that  $m_1>6$.  

Since $\Theta_1>0$ it follows $r \geq 3$  (so $r=3$) and $\Theta_1 +\frac{1}{m_1} \geq \frac16$ with equality if and only if the signature is $(0;2,3,m_1)$.
So $\Theta_1\geq \frac{m_1-6}{6m_1}$ and
$$m_1 \leq 2\left(\frac{\xi}{\Theta_1}+3\right)  \leq 2\left(\frac{6m_1\xi}{m_1-6}+3\right)$$
which is equivalent to
$$m_1^2-12\left( \xi+1 \right)m_1+36\leq 0$$
and we can conclude as before.
\end{proof}

We are now prepared to give the necessary bounds in order to show that given $\gamma$, $p_g$ and $q$, there is a finite number of product-quotient surfaces with these invariants.

 Recall that $g(C_i/G)$, $i=1,2$, is bounded by $q$ (Proposition \ref{forminf}), whence it is enough to produce upper bounds for the remaining natural numbers involved, i.e., we need to bound $r,s$, $m_i$ and $n_j$ in terms of $p_g$ and $q$. t
\begin{rem}\label{boundcardbasket} If $S$ is of general type then
$$
\# \Sing X= \# \mathfrak{B}(X)  \leq 8\chi+4\gamma-1.
$$
\end{rem}
\begin{proof}
The  inequality follows by Corollary \ref{4chi+gamma-mu>=1/2} since, by the definition of $\mu$, $\# \mathfrak{B}(X) \leq 2 \mu$.
\end{proof}

We give now an upper bound for the multiplicity of each singularity of $X$ in terms of $p_g$, $q$ and $\gamma$. This, together with remark \ref{boundcardbasket}, produces a finite list of possibilities for the basket of singularities of the quotient model of a product quotient surface with given values of $p_g$, $q$ and $\gamma$.

\begin{prop}\label{boundmult} If $S$ is of general type then
\begin{itemize}
\item[a)] If $\frac{q}{n} \in {\mathfrak B}$, then $n\leq12(4\chi+2\gamma-1)$;
\item[b)] if moreover $\gamma\neq 0$, then $n\leq12(4\chi+2\gamma-\frac32)$.
\end{itemize}
\end{prop}
\begin{proof}
a) If the basket is empty, then the claim is empty. Otherwise assume that there is a singular point $x$ of type  $\frac{q}{n}$, and let $m_i$ be the multiplicity of the central fibre of $f_1$ containing it.Then by lemma
\ref{PolizziInt} there is at least one further singular point on the same fibre, and, if there is only one, it is of type $\frac{n-q}{n}$. It follows $\mu \geq 2-\frac{2}{n}$. 

By proposition \ref{bound} and remark \ref{ndividesmi} then $n \leq m_i<12(4\chi+2\gamma-\mu+1)\leq 12(4\chi+2\gamma-1+\frac2n)$. Therefore
$$ n-\frac{24}{n}<12(4\chi+2\gamma-1).$$

If $4\chi+2\gamma-1\geq 2$, then the righthand side is bigger than $24$, hence 
\begin{equation}\label{gamma=0:n<=}
n\leq12(4\chi+2\gamma-1).
\end{equation}

By Proposition \ref{gamma+pg} and Corollary \ref{4chi+gamma-mu>=1/2}, $4\chi+2\gamma-1$ is a positive integer, so it remains to consider only the case  $4\chi+2\gamma-1=1$. In this case Corollary \ref{4chi+gamma-mu>=1/2} yields $\mu \leq \frac32$, and therefore either there are three points of multiplicity $2$ or there are exactly two singular points, both of multiplicity $n \leq 4$. In all cases (\ref{gamma=0:n<=}) hold.

b) If the basket contains exactly $2$ elements, they are by Proposition \ref{PolizziInt} of respective type $\frac{q}{n}$ and  $\frac{n-q}{n}$ and then by Proposition \ref{gammasym} $\gamma=0$. Therefore
$\gamma \neq 0$ implies that there are at least three singular points, and
a straightforward computation gives $\mu \geq \frac52-\frac{3}{n}$, whence 
$$n-\frac{36}{n}< 12\left(4\chi+2\gamma-\frac32 \right).$$

The claim follows by the same argument as in the previous case.
\end{proof}

\begin{rem}
We have shown that the classification problem is finite. In fact, we know that there are finitely many possibilities for the basket of singularities. If we fix a basket $\mathfrak{B}$, then we have to show that there are finitely many possibilities for
\begin{itemize}
\item the order of the group $G$,
\item the two types $t_1=(g(C_1/G);m_1,\ldots , m_r)$ and $t_2=(g(C_2/G);n_1,\ldots , n_s)$.
\end{itemize}
Note that by proposition \ref{robavecchia}, a), $|G|$ is determined by $t_1$ and $t_2$. The length $r$ (resp. $s$) of $t_1$ (resp. $t_2$) is bounded by proposition \ref{bound}, a), whereas a bound for the $m_i$ (resp. $n_j$) is given y loc.cit. b), c).
\end{rem}

We have now enough elements to write an algorithm producing, for each fixed value of the triple 
$(p_g, q, \gamma)$, all product quotient surfaces with those values of $p_g$, $q$ and $\gamma$. Still, for implementing a reasonable (quick) algorithm it is convenient to use also the  following additional informations which we have proved in \cite{bp}.

\begin{prop}\label{robavecchia}\ 
\begin{itemize}
\item[a)] $|G|=\frac{4\alpha_1\alpha_2}{\xi}=\frac{\xi}{\Theta_1\Theta_2}$;
\item[b)] for each $i$,  $ \frac{I\xi}{\Theta_1 m_i} \in \NN$; 
\item[c)] there are at most $\frac{|{\mathfrak B}|}{2}$ indices such that 
$\frac{I \xi}{2\Theta_1 m_i} \not\in \NN$; 
\item[d)] $m_i\leq \frac{1+I\xi}{f}$, where $f:=\max(\frac 16, \frac{r-3}{2})$;
\item[e)] except for at most $\frac{|{\mathfrak B}|}{2}$ indices, it holds: $m_i\leq \frac{2+ I\xi}{2f}$
\end{itemize}
Similar statements as b), c), d) obviously hold for $(n_1, \ldots , n_s)$.
\end{prop}

\begin{proof}
a) follows by Remark \ref{xipq} and Proposition \ref{galfa};

b-c) see \cite{bp}, proposition 1.13;

d) let $m_1$ be the biggest of the $m_i$; then $\Theta_1+\frac{1}{m_1} \geq f$ whence:
$$
m_i \leq m_1 \leq \frac{1+\Theta_1m_1}f \leq \frac{1+I\xi}f;
$$
where $f:=\max(\frac 16, \frac{r-3}{2})$;

e) similar.
\end{proof}

We can now describe explicitly an algorithm producing all product quotient surfaces of general type with fixed $p_g$, $q$ and $\gamma$.

Indeed, Corollary \ref{boundcardbasket} and Proposition \ref{boundmult} produce, once fixed $p_g$, $q$ and $\gamma$, a finite list of possible baskets. The basket determines also $\mu$, $l$ and $\xi$.

Moreover,  $0 \leq g(C_1/G) \leq q$ varies also in a finite set (and determines $g(C_2/G)=q-g(C_1/G)$).

For each basket in the list, and for each choice of $g(C_1/G)$, Proposition \ref{bound} gives a finite list of possible signatures for the action of $G$ on $C_1$ (and similarly on $C_2$).
Most of the signature obtained can be excluded by using the other conditions we know: 
\begin{itemize}
\item Remark \ref{ndividesmi} ensures that each signature contains a multiple of the multiplicity of each singularity;
\item $\alpha \in \NN$;
\item Proposition \ref{robavecchia}, b), c), d), e).
\end{itemize}

Finally, for each pair of signatures, we can run a search on all groups of the order predicted by Proposition \ref{bound}, a), for pair of generating vectors of the prescribed signatures.

We have implemented this algorithm in MAGMA (cite{magmaref}) in the case $q=0$, the interested reader may download the commented script from

 \url{http://www.science.unitn.it/~pignatel/papers/RegP-QByPgGamma.magma}

The command {\em ExistingSurfaces($p_g$,$\gamma$,$M$)} produces two outputs: a list of regular product-quotient surfaces with the given values of $p_g$ and $\gamma$, and quotient model whose singularity of maximal multiplicity has multiplicity $M$, and a list of {\it skipped} cases, pairs (group,signature) 
which the computer could not compute (for technical reasons: if there is a regular product-quotient surface with those values of $p_g$ and $\gamma$ which is not in the first output, group and signature are in the second output. 

To get all surfaces product-quotient surface with given values of $p_g$ and $\gamma$ one should run it with $M$ up to the maximum predicted in Proposition \ref{boundmult}, and then check the second output for missing surfaces. In all cases we run we could show, by argument similar to those used in \cite{bp}, that the first list is complete; in other words, that the computation skipped by the computer do not give any surface.

\section{$\gamma$ detects minimality?}

In \cite{bp} the authors ran a computer program whose output lists all product-quotient surfaces with $p_g=0$ and $K_S^2 \geq 1$. Inspecting the output it turned out that all surfaces but one are minimal (hence of general type). All the minimal product-quotient surfaces satisfy $\gamma(S) = 0$, while the only non-minimal surface in the list has $\gamma = 1$. It seems therefore natural to conjecture that $\gamma$ is related to the minimality of a product-quotient surface. Or, more ambitiously, that one can bound the number of exceptional (-1)-cycles on a product-quotient surface in terms of $\gamma$.

We pose the following 
\begin{conj}

Let $S$ be a regular product-quotient surface of general type. Then 
$$\gamma(S) + p_g(S) = 0 \iff S \ \rm{is} \ \rm{minimal}.$$

\end{conj}

In the sequel we shell give a proof of this conjecture in the special case $p_g=0$. In fact, we have
\begin{theo}\label{gamma0}
Let $S$ be a regular product-quotient surface of general type with $p_g=0$. Then 
$$\gamma(S)= 0 \iff S \ \rm{is} \ \rm{minimal}.$$
\end{theo}

\begin{rem}
Unfortunately, we do not have a conceptional proof of the above theorem, which could shed some light on a possible connection between the number of exceptional cycles on a product-quotient surface and the invariant $\gamma$, or $\gamma+ p_g$. The proof is just a case by case inspection of the output of the MAGMA script listing all product-quotient surfaces with $p_g = \gamma = 0$.
\end{rem}
\begin{proof}
Running the MAGMA script "ExistingSurfaces(0,0,M)" with $M$ up to $36$, we only have to take care of the surfaces $S$ with $K_S^2 \leq 0$. In fact, if $K_S^2 >0$, it has already been proven in \cite{bp} (cf. also Theorem \ref{classiso} and the corresponding tables) that in these cases $\gamma=0$. 

Therefore the proof is finished once we show that the cases with $K_S \leq 0$ in the output of "ExistingSurfaces(0,0,M)" with $M$ up to $36$ are not of general type.

This will be taken care of in the remaining part of the section.
\end{proof}

First of all we list the output of the surfaces with $K_S^2 \leq 0$ in  table \ref{K2<1}.
\begin{table}
\caption{Product-quotient surfaces  with $\gamma = p_g=0$ not of general type}
\label{K2<1}
\renewcommand{\arraystretch}{1,3}
 \begin{tabular}{|c|c|c|c|c|c|c|c|}
\hline
&$K^2_S$&Sing X&$t_1$&$t_2$&$G$\\
\hline\hline
1)&   0&$\frac 16, \frac 56,2\times \frac 12$ &$2,4,6$&$2,4,6$& SmallGroup(192,955)  \\
   \hline
  2)& 0&$\frac 16, \frac 56,2\times \frac 12$ &$2,4,6$&$2,5,6$& SmallGroup(120,34)  \\
\hline
 3)&  0&$\frac 16, \frac 56,2\times \frac 12$ &$2,4,6$&$2,2,2,6$& SmallGroup(48,48)  \\
\hline
 4)&  -2&$2\times \frac 15,2\times \frac 45$ &$2,5,5$&$2,5,5$& SmallGroup(80,49)  \\
\hline
 5)&  0&$4\times \frac 25$ &$2,5,5$&$2,5,5$& SmallGroup(80,49)  \\
\hline
 6)&  0&$2\times \frac 14,2\times \frac 34$ &$2,4,5$&$3,4,4$& SmallGroup(120,34)  \\
\hline
 7)&  0&$2\times, \frac 14,2\times\frac 34$ &$2,2,2,4$&$2,2,2,4$& SmallGroup(16,11)  \\
\hline
 8)&  0&$2\times \frac 14,2\times \frac 34$ &$2,2,2,4$&$3,4,4$& SmallGroup(24,12)  \\
\hline
 9)&  0&$2\times \frac 14,2\times \frac 34$ &$3,4,4$&$3,4,4$& SmallGroup(36,9)  \\
\hline
 10)&  -1&$\frac 15, 2\times \frac 25, \frac 45$ &$2,5,5$&$3,3,5$& SmallGroup(60,5)  \\
\hline
\end{tabular}
\end{table}

We need the following:

\begin{prop}\label{E2KE}
Let $S$ be a product-quotient surface and let $A_1, \ldots A_l$ be the exceptional curves of $\sigma$ of respective selfintersection $b_i$. Assume that 
$$
E \sim \frac{\mu_1}{|G|}F_1 + \frac{\mu_2}{|G|}F_2 - \sum_{i=1}^l a_iA_i \in H^2(S,\QQ).
$$  
Then $\mu_i \in \NN$. Moreover, let $M$ be the intersection matrix of the basket (i.e., of the $A_i$'s), and set
$$b:=\begin{pmatrix}K_SA_1\\K_SA_2\\ \cdot \\ \cdot \\ \cdot \\ K_SA_l \end{pmatrix}=\begin{pmatrix} b_1-2\\b_2 -2\\ \cdot \\ \cdot \\ \cdot \\ b_l -2\end{pmatrix}, \  \  e:=\begin{pmatrix} E A_1\\E  A_2\\ \cdot \\ \cdot \\ \cdot \\ E  A_l \end{pmatrix}.$$
Then 
\begin{equation}\label{KE}
K_SE = \mu_1 \Theta_2 + \mu_2 \Theta_1 + e^T M^{-1}b;
\end{equation}
\begin{equation}\label{E2}
E^2 =  \frac{2\mu_1\mu_2}{|G|} + e^T M^{-1}e.
\end{equation}
\end{prop}

\begin{proof} Note that $\mu_1=EF_2$; in particular $\mu_1 \in \NN$. Similarly $\mu_2 \in \NN$.

Since
$$
\sigma^* K_X \equiv \frac1{|G|} \left( (2g_1-2)F_1 + (2g_2-2)F_2\right) \equiv \Theta_1F_1 + \Theta_2 F_2,
$$
then $K_S \equiv \Theta_1F_1 + \Theta_2 F_2 - \mathfrak{A}$, where $\mathfrak{A}$ is of the form $\sum_{i=1}^l \alpha_iA_i$ for some $\alpha_i \in \QQ$. Set
$$
a:=\begin{pmatrix} a_1\\a_2 \\ \cdot \\ \cdot \\ \cdot \\ a_l \end{pmatrix},\
\alpha:=\begin{pmatrix} \alpha_1\\ \alpha_2 \\ \cdot \\ \cdot \\ \cdot \\ \alpha_l \end{pmatrix}.
$$
Since $\forall i$ $A_iF_1 = A_iF_2 =0$, then  $\mathfrak{A}A_i = -K_S A_i=-(b_i-2)$; in other words $M \alpha=-b$. Similarly $Ma=-e$. Since $M$ is invertible, we can also write $b=-M^{-1}\alpha$, $e=-M^{-1}a$.

Then $K_SE = \mu_1 \Theta_2 + \mu_2 \Theta_1 +\sum a_iA_i \mathfrak{A} = \mu_1 \Theta_2 + \mu_2 \Theta_1 +a^TM\alpha= \mu_1 \Theta_2 + \mu_2 \Theta_1 +e^TM^{-1}b$. Similarly 
$E^2 =  \frac{2\mu_1\mu_2}{|G|} - \left( \sum_{i=1}^l a_iA_i \right)^2=  \frac{2\mu_1\mu_2}{|G|} +a^TMa=  \frac{2\mu_1\mu_2}{|G|} +e^TM^{-1}e.$
\end{proof}

\begin{rem}
By the proof of Proposition \ref{gamma+pg}, if $p_g+\gamma=0$, the set $\{A_i,F_j\}$ is a basis of $H^2(X,\QQ)$, so the assumption of Proposition \ref{E2KE} is automatically verified by every curve $E$.
\end{rem}

To show that the surfaces in table \ref{K2<1} are not of general type we argue by cotradiction, assuming that they are of general type, and showing that the minimal model has 
$K_{\overline{S}}^2<0$. To do that, we look for rational curves $E$ with selfintersection $-1$, and study their image $\sigma(E)$ in the quotient model $X$. 

We recall that
\begin{prop}\label{3points}
Let $\PP^1 \rightarrow X$ be a rational curve in $X$, so a map which is generically injective. Then the preimage of the singular locus of $X$ has cardinality at least three. 
\end{prop}
\begin{proof}
This has been shown in the proof of \cite[Proposition 4.7]{bp}
\end{proof}

\begin{prop}\label{KC<0}
Every irreducible curve $C$ on a smooth surface $S$ of general type with $K_SC\leq 0$ is smooth and rational. 
\end{prop}
\begin{proof}
See \cite [Remark 4.3]{bp}
\end{proof}

\begin{cor}\label{b-3}
If $S$ is a surface of general type, $E$ a $(-1)-$curve on $S$ and $C$ a curve with $C^2=-b$. Then
$CE \leq\max (1, b-3)$. 
\end{cor}
\begin{proof}
Else, contracting $E$, we obtain a surface of general type with a curve, the image of $C$, contradicting Proposition \ref{KC<0}.
\end{proof}
We will also need

\begin{lemma}\label{-4-3-2}
Let $S$ be a product-quotient surface of general type. Suppose that the exceptional locus of $\sigma$ consists of
\begin{itemize}
\item[i)] curves of self intersection (-3) and (-2), or
\item[ii)]at most two smooth rational curves of self-intersection (-3) or (-4), and (-2)-curves.
\end{itemize}
Then $S$ is minimal.
\end{lemma}
\begin{proof}
i) This is \cite[Corollary 4.8]{frapi}.

ii) Assume that $S$ contains a (-1)-curve $E$. Note that $E$ cannot intersect two different (-2)-curves or, contracting it, we would get two (-1)-curves intersecting transversally, impossible on a surface of general type. Then by Proposition \ref{3points} and Corollary \ref{b-3} the exceptional locus contains two curves of self-intersection (-3) or (-4), say $E_1$ and $E_2$, $EE_1=EE_2=1$ and moreover $E$ intersects exactly one (-2)-curve, transversally. After contracting $E$, then the image of the (-2)-curve we get two rational curves of  self intersections (-1) or (-2), intersecting each other with multiplicity bigger than one, which is impossible on a surface of general type.
\end{proof}

\begin{rem}\label{singular}

Observe that if we arrive, after contracting one or more exceptional curves, to a configuration as in the previous lemma with maybe singular (-4) resp- (-3)-curves, the same argument applies, showing that on a surface of general type there cannot be more (-1)-curves.
\end{rem}

We can now prove that all surfaces in table \ref{K2<1} are not of general type.

\begin{lemma}
The product-quotient surfaces 1), 2), 3) in table \ref{K2<1} are not of general type.
\end{lemma}
\begin{proof}
In this case the basket is $\{\frac16, \frac 56,2\times \frac 12 \}$. We have $8$ curves $A_i$, which we order in a natural way, so $A_1^2=-6$, and $A_7$ and $A_8$ are the curves produced by the nodes.

Assume that there exist a (-1)-curve $E$. 
 In the notation of Proposition \ref{E2KE}
$$b = \begin{pmatrix} 4\\0\\ \cdot \\ \cdot \\ \cdot \\ 0\end{pmatrix}.$$
For what concerns $e$ we notice again that $E$ cannot intersect two different (-2)-curves, so by Proposition \ref{3points} and corollary \ref{b-3}, $BA_1 \geq 2$. But then $E$ cannot intersect $A_2, \ldots, A_6$ since else, after contracting it we could contract enough other curves intersecting the image of $A_1$ to contradict proposition \ref{KC<0}. The possibilities left for $e$ are
$$
a) \begin{pmatrix} 3\\  0\\ \cdot \\ \cdot \\ \cdot \\ 0\end{pmatrix}, \ \ b) \begin{pmatrix} 2\\0\\ \cdot \\ \cdot  \\0 \\ 1\end{pmatrix}, \ \ c) \begin{pmatrix} 2\\0\\ \cdot \\ \cdot \\ 1 \\ 0 \end{pmatrix}.
$$
Note that the second and third case are symmetric, one obtained from the other exchanging the two nodes. Therefore it suffices to treat only the cases a) and b).

Then applying proposition \ref{E2KE} and substituting $EK_S = E^2=-1$ in equations \ref{E2}, \ref{KE} and solving respect to the variables $\mu_i$
we get in each of the three cases:
\begin{enumerate}
\item Here $\Theta_1 = \Theta_2 = \frac{1}{12}$ and we get:
\subitem a) $\mu_1 +\mu_2 = 12$, $\mu_1\mu_2 = 48$;
\subitem b) $\mu_1 +\mu_2 = 4$, $\mu_1\mu_2 = 16$.

\item Here $\Theta_1 = \frac{1}{12}$, $\Theta_2 = \frac{2}{15}$                                                                                                                                                                                                                   and we get:
\subitem a) $8\mu_1 +5\mu_2 = 60$, $\mu_1\mu_2 = 30$;
\subitem b) $8\mu_1 +5\mu_2 = 20$, $\mu_1\mu_2 = 10$.

\item Here $\Theta_1 = \frac{1}{12}$, $\Theta_2 = \frac 13$                                                                                                                                                                                                                   and we get:
\subitem a) $4\mu_1 +\mu_2 = 12$, $\mu_1\mu_2 = 12$;
\subitem b) $4\mu_1 +\mu_2 = 4$, $\mu_1\mu_2 = 4$.
\end{enumerate}
All systems have no integral solutions, a contradiction.
\end{proof}

\begin{lemma}
The product-quotient surface 4) in table \ref{K2<1} is not of general type.
\end{lemma}

\begin{proof}Here the basket is  $\{2\times \frac 15,2\times \frac 45 \}$. Assume that $S$ is of general type. Since $K_S^2 = -1$ there must be a (-1)-curve $E$ on $S$. $E$ has to intersect t least a (-5)-curve, and cannot intersect any rational (-2)-curve (or, as in the previous proof, after contracting it, we could contract enough curves to contradict Proposition \ref{KC<0}).

So $E$ passes twice through one of the (-5)-curves and at least once through the other. After contracting $E$ we get a surface $S'$ with a configuration of rational curves as in Remark \ref{singular}. Therefore we conclude that $S'$ is minimal, contradicting $K_{S'}^2 = 0$.
\end{proof}

\begin{lemma}\label{5-9}
The product-quotient surfaces 5),  6), 7), 8), 9) in table \ref{K2<1} are not of general type.
\end{lemma}
\begin{proof}
Here the basket is $\{2\times \frac 14,2\times \frac 34 \}$ or 
 $\{ 4\times \frac 25 \}$. In all cases, if $S$ is of general type, ity is minimal by Lemma \ref{-4-3-2}, contradicting $K_S^2=0$.
\end{proof}

\begin{lemma}
The product-quotient surface 10) in table \ref{K2<1} is not of general type.
\end{lemma}

\begin{proof} Here the basket is $\{\frac 15, 2\times \frac 25, \frac 45 \}$.
Assume that $S$ is of general type. Then $S$ contains a (-1)-curve $E$. After contracting $E$, which has to pass at least once through the (-5)-curve and at least once through a (-3)-curve, we get a surface $S'$ with a configuration of rational curves as in Remark \ref{singular}. Therefore we conclude that $S'$ is minimal, contradicting $K_{S'}^2 = -1$.
\end{proof}

This concludes the proof of Theorem \ref{gamma0}.
\section{Surfaces of general type with $p_g=0$ and $\gamma>0$}\label{constructions}

We shall give now a detailed description of the minimal models of the three product-quotient surfaces  of general type with $p_g=0$ and $\gamma >0$ which we were able to find running our computer program. In fact, we believe that there are no more non-minimal product-quotient surfaces  of general type with $p_g=0$ left.
\subsection{A numerical Godeaux surface with torsion of order $4$}
The group $G$ is the subgroup of order $96$ of the permutation group $S_8$ generated by $(123)$, $(12)(34)$, $(57)$ and $(5678)(12)$.

Its action on $\{1,\ldots,8\}$ has two orbits, $\{1,\ldots,4\}$ and $\{5,\ldots,8\}$. Indeed $G$ is an index $2$ subgroup of $S_4 \times D_4$ where $S_4$ is the permutation group of $\{1,2,3,4\}$, and $D_4$ is the isometry group of the square, embedded in $S_8$  by considering its action on the vertices of the square and labeling them counterclockwise as $5,6,7,8$.

The curves $C_1$ and $C_2$ are very similar, they are both G-covers of $\PP^1$ branched on $\{p_1=1, p_2=0, p_3=\infty\}$ with respective generating vectors
\begin{itemize}
\item $\{g_1:=(123)(57),g_2:=(4321)(56)(78),g_3:=(g_1g_2)^{-1})\}$;
\item $\{g'_1:=(123)(57),g'_2:=(4321)(5678),g'_3:=(g'_1g'_2)^{-1})\}$.
\end{itemize} 
Their respective signatures are $(0;6,4,4)$ and $(0;6,4,2)$.

Our computer program shows that
\begin{prop}
The product-quotient surface $S$ with quotient model $X=(C_1 \times C_2)/G$ above has $p_g=q=K^2_S=0$, $\pi_1(X)\cong \ZZ_4$ and $\gamma=1$. The basket of singularities of 
$X$ is $\{2\times\frac16,\frac23,2\times \frac12\}$. All singular points of $X$ are mapped onto $(1,1)$ by the natural map $X=(C_1 \times C_2)/G \rightarrow \PP^1 \times \PP^1= C_1/G \times C_2/G$.
\end{prop}

We consider the map $\PP^1 \stackrel{z \mapsto z^4}{\longrightarrow}\PP^1$, and the normalization of the fibre product as in the following commutative diagram (for $i=1,2$):
\begin{equation*}
\xymatrix{
C'_i\ar^{\zeta_i}[r]\ar^{\lambda_i}[d]&C_i\ar[d]\\
\PP^1\ar[r]&\PP^1 \\
}
\end{equation*}
Then $\lambda_i$ is a $G-$cover of $\PP^1$ branched in the $4-$th roots of unity. Lifting loops as in the following picture

\includegraphics[width=\linewidth]{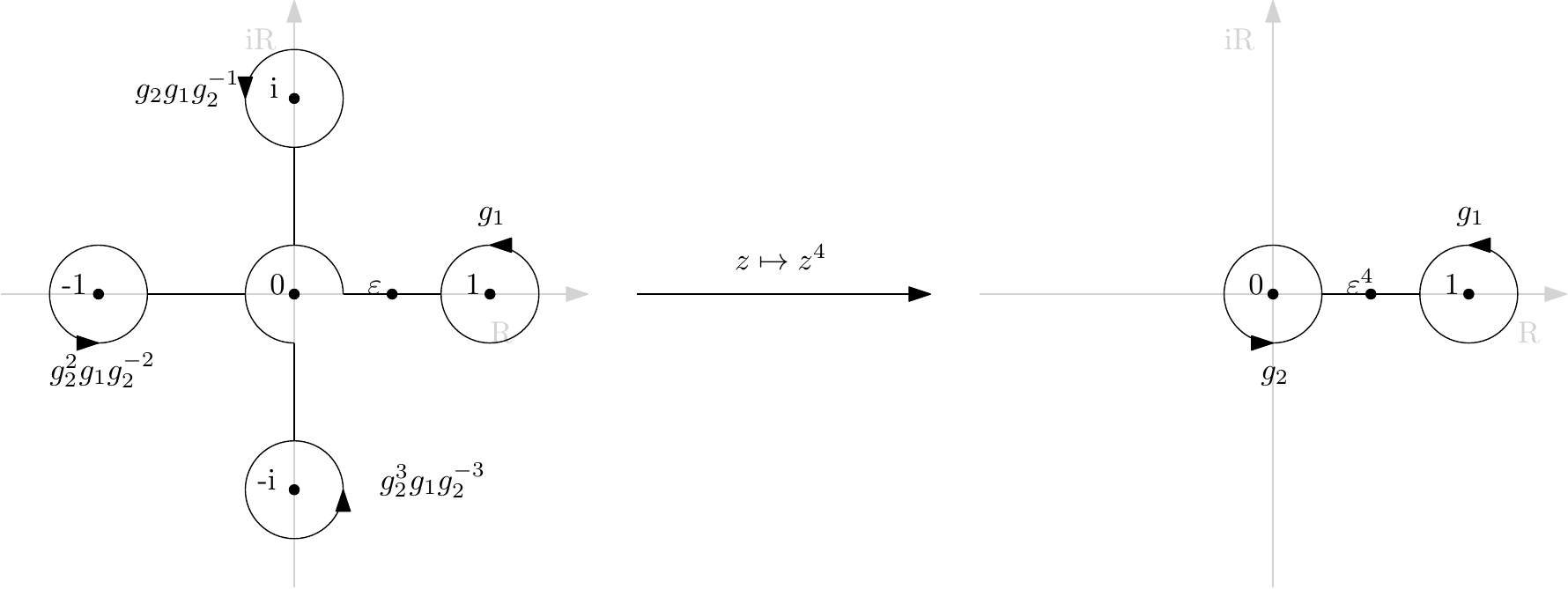}
we see that $\lambda_1$ is the $G-$cover with generating vector $\{g_1,g_2g_1g_2^{-1},g_2^2g_1g_2^{-2},g_3g_1g_2^{-3}\}$
and  $\lambda_2$ is the $G-$cover with the analogous generating vector obtained substituting $g_i$ with $g_i'$.

\begin{rem}
It is worth mentioning that here the word "generating" is a slight abuse of notation,  since the above elements do not generate the whole group $G$. This  implies that $C'$ is not connected, the number of connected components being the index of the  subgroup generated by $\{g_1,g_2g_1g_2^{-1},g_2^2g_1g_2^{-2},g_3g_1g_2^{-3}\}$ in $G$; this does not affect in any way our argument.
\end{rem}
The reader can easily check that the two generating vectors coincide, so $\lambda_1$ and $\lambda_2$ are isomorphic $G-$covers. In particular, we have a map 
$
\zeta':\Gamma \cong C_1' \cong C_2' \stackrel{(\zeta_1,\zeta_2)}\longrightarrow C_1 \times C_2
$
which is $G-$equivariant, henceinduces a map on the quotient
$$
\zeta \colon \Gamma/G \cong \PP^1 \rightarrow X=(C_1 \times C_2)/G; 
$$
$E':=\zeta(\PP^1)$ is a rational curve on $X$.

Denote by $A_1$, $A_2$ the exceptional curves produced by the singularities $\frac16$, and by $E$ the strict transform of $E'$. 
\begin{prop}
$E$ is a smooth rational curve with $K_SE=E^2=-1$. Moreover, $E(A_1+A_2)=4$, and $EA_i=0$ for every further exceptional curve $A_i$ of $\sigma$.
\end{prop}
\begin{proof}
First of all, let us show that $\zeta$ is generically injective. In fact, composing with the map $X \rightarrow \PP^1 \times \PP^1$ we get the map
$\PP^1 \stackrel{z \mapsto z^4}{\longrightarrow} \PP^1 \times \PP^1$; this shows that $\zeta$ is d-to-1 for a positive integer d which is a divisor of 4.

On the other hand, since all singular points of $X$ lie over $(1,1)$, this also shows that only the $4^{th}$ roots of unity may be mapped to singular points of $X$. 
So $E'$ will pass at most $\frac4d$ times through singular points of $X$; by Proposition \ref{3points}, $d=1$.

The smoothness of $E$ follows easily by a local computation. The only points of $E'$ contained in $\Sing X$, the $4^{th}$ roots of unity, have stabilizer of order $6$, so they are mapped to singular points of multiplicity $6$. This implies $E(A_1+A_2)=4$ and $EE_i=0$ for every further exceptional curve. 

Then $E=\sigma^*E'-\frac{a_1}6 A_1-\frac{a_2}6 A_2$ with $a_1+a_2=4$. Moreover $K_XE'=\frac{K_{C_1 \times C_2}\zeta'(\Gamma)}{|G|}=\frac{4|G|\Theta_1+4|G|\Theta_2)}{|G|}=4(\Theta_1+\Theta_2)=\frac53$. Therefore
$$
K_SE=K_XE'-\frac{a_1}6 K_SA_1-\frac{a_2}6 K_SA_2=\frac53-\frac46 (a_1+a_2)=-1.
$$
\end{proof}

Finally we can prove 
\begin{theo}
Contracting $E$ we get a minimal surface. In particular the minimal model of $S$ is a numerical Godeaux surface with torsion of order $4$.
\end{theo}
\begin{proof}
Since $K_{\overline S}^2>0$, $q=0$, $\pi_1(\overline{S}) \neq 0$,  by the Enriques-Kodaira classification $\overline{S}$ is of general type.

By corollary \ref{b-3}, $EA_1,EA_2 \leq 3$, so $(EA_1,EA_2)$ equals either $(2,2)$, or $(3,1)$, or $(1,3)$.
 
In the first case, the following picture describes how the configuration of curves changes after the contraction.

\includegraphics[width=\linewidth]{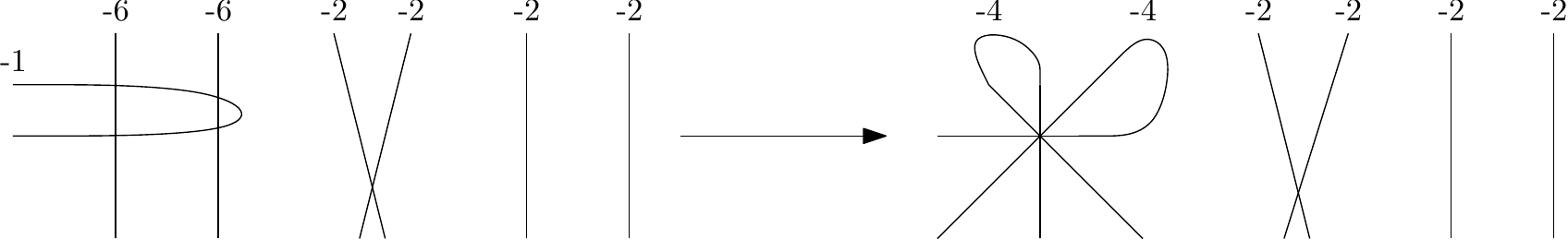}

The minimality follows then directly by remark \ref{singular}. 
A similar argument gives the minimality in the other two cases.
\end{proof}

\subsection{A numerical Godeaux surface with torsion of order $5$}
The group $G$ is $\ZZ_5^2$. The curves $C_1$ and $C_2$ are two G-covers of $\PP^1$ branched on $\{p_1=1, p_2=0, p_3=\infty\}$ with respective generating vectors
\begin{itemize}
\item $\{g_1:=(1,0),g_2:=(0,1),g_3:=(g_1g_2)^{-1})\}$;
\item $\{g'_1:=(1.0),g'_2:=(1,1),g'_3:=(g'_1g'_2)^{-1})\}$.
\end{itemize} 
Both signatures are $(0;5,5,5)$.
Our computer program shows that
\begin{prop}
The product-quotient surface $S$ with quotient model $X=(C_1 \times C_2)/G$ above has $p_g=q=K^2_S=-1$, $\pi_1(X)\cong \ZZ_5$ and $\gamma=2$. The basket of singularities of 
$X$ is $\{5\times\frac15\}$. All singular points of $X$ are mapped onto $(1,1)$ by the natural map $X=(C_1 \times C_2)/G \rightarrow \PP^1 \times \PP^1= C_1/G \times C_2/G$.
\end{prop}

Since all singularities lie over $(1,1)$, they all lie in the same fibre of each of the two isotrivial fibrations, whose central fibres we denote respectively by $E_1$, $E_2$, which are $\ZZ_5$-quotients of  $C_2$ resp. $C_1$ with $5$ branching points; Hurwitz' formula shows that both $E_1$ and $E_2$ are rational. By Proposition \ref{PolizziInt} $E_1^2=E_2^2=-5\frac15=-1$. So both curves are exceptional divisors of the first kind. 

\begin{theo}
Contracting $E_1$ and $E_2$ we get a minimal surface. In particular, the minimal model of $S$ is a numerical Godeaux surface with torsion of order $5$.
\end{theo}
\begin{proof}
Since $K_{\overline S}^2>0$, $q=0$, $\pi_1(\overline{S}) \neq 0$, by the Enriques-Kodaira classification $\overline{S}$ is of general type.

The following picture describes how the configuration of curves changes after the contraction.

\includegraphics[width=\linewidth]{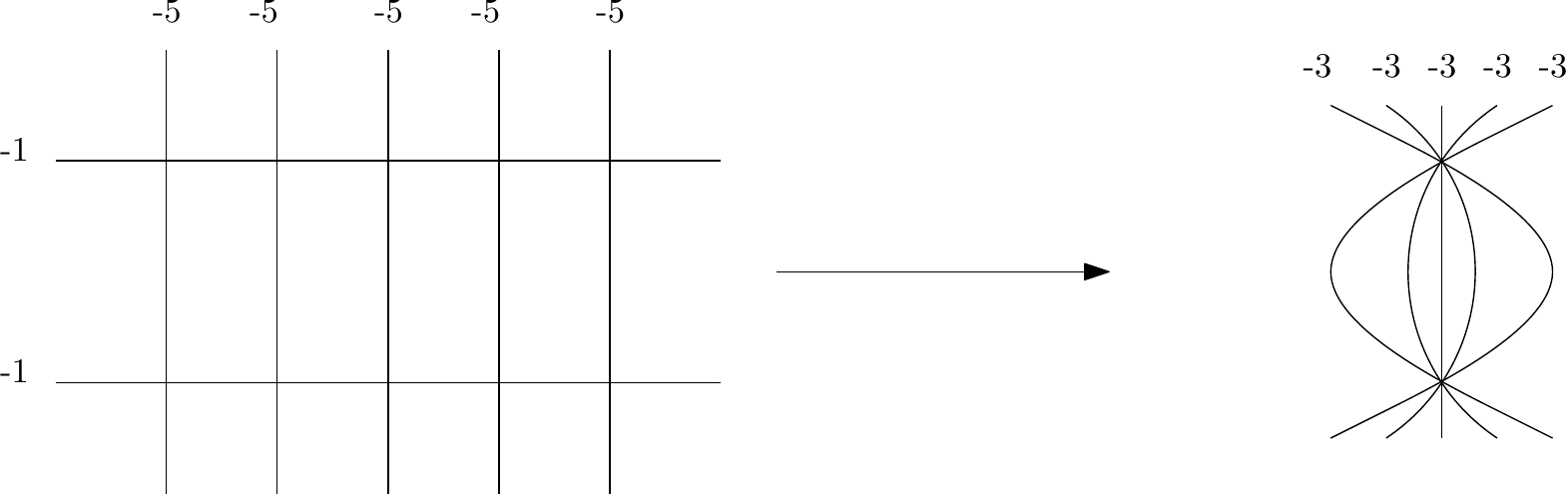}

The minimality follows then directly by remark \ref{singular}.
\end{proof}

\subsection{Are there more product-quotient surfaces of general type with $p_g=0$?}

By the results in\cite{bp} and Theorem \ref{gamma0} there are exactly 72 families of
surfaces of general type with $p_g=\gamma=0$. By Proposition \ref{gamma+pg} all missing 
product-quotient surfaces of general type have $\gamma>0$. We know three examples of them, the 
{\it fake Godeaux} described in \cite{bp} (with $K_S^2=\gamma=1$, $K_{\bar{S}}^2=3$), and the two numerical Godeaux surfaces described in this section.

We can prove the following
\begin{prop}\label{nothingmore}
Let $S$ be a product-quotient surface of general type with $p_g=0$ not among the $75$ families just mentioned.
Then

\begin{itemize}
\item either $\gamma \geq 4$,
\item or $\gamma=3$ and $X$ has a singular point of multiplicity at least $14$,
\item or $\gamma=2$ and $X$ has a singular point of multiplicity at least $45$.

\end{itemize}
\end{prop}

The proof is obtained by running our program for $\gamma=1$ and multiplicity up to $54$, 
$\gamma=2$ and multiplicity up to $44$, $\gamma=3$ and multiplicity up to $13$, 
and then by showing case by case that the resulting surface is not of general 
type. 

The full list of the cases to consider is the tables \ref{gamma1},  \ref{gamma2}, \ref{gamma2b} and  \ref{gamma3}. Note that in the last column we list only the SmallGroup identifier of the MAGMA database of groups up to order 2000, i.e. (n,m) means the m-th group of order n.

\begin{table}
\caption{Product-quotient surfaces not of general type with $p_g=q=0$, $\gamma=1$}
\label{gamma1}
\renewcommand{\arraystretch}{1,3}
 \begin{tabular}{|c|c|c|c|c|c|c|c|}
\hline
$\gamma$&$K^2_S$&Sing X&$t_1$&$t_2$&$G$\\
\hline\hline
1&  -2&$4 \times \frac 12,4\times \frac 14$ &$4,4,4$&$4,4,4$& (16,2) \\
   \hline
1&  -3&$2 \times \frac 12, \frac 13, 2 \times \frac 23,2\times \frac 16$ &$2,6,6$&$2,6,6$& (48,49)  \\
   \hline
1&  -3&$4 \times \frac 12,\frac 17,2\times \frac 27$ &$2,3, 7$&$4,4, 7$& (168,42)  \\
   \hline
1&  -3&$4 \times \frac 12, \frac 14, \frac 18, \frac 58$ &$2,4,8$&$4,4,8$& (32,11)  \\
   \hline
1&  -4&$6 \times \frac 12,\frac 23,2 \times \frac 16$ &$2,4,6$&$2,2,2,6$& (24,8)  \\
   \hline
1&  -4&$2 \times \frac 13,3 \times \frac 23,2\times \frac 16$ &$3,3,6$&$3,3,6$& (36,11)  \\
   \hline
1&  -4&$2 \times \frac 12,\frac 13, \frac 23, \frac 17, 2\times \frac 27$ &$2,3,7$&$3,4,7$& (168,42)  \\
   \hline
1&  -4&$7 \times \frac 12,\frac 18,\frac 38$ &$2,3, 8$&$2,2,2,8$& (48,29) \\
   \hline
1&  -4&$2 \times \frac 12,\frac 13 , \frac 23, \frac 14, \frac 18, \frac 58$ &$2,3,8$&$3,4,8$&(96,64)  \\
   \hline
1&  -5&$2 \times \frac 13, 2 \times \frac 23,\frac 17,2\times \frac 27$ &$2,3,7$&$3,3,7$& (168,42)   \\
   \hline
1&  -5&$2 \times \frac 13,2 \times \frac 23,\frac 17,2\times \frac 27$ &$3,3,7$&$3,3,7$& (21,1)  \\
   \hline
1&  -5&$2 \times \frac 12,2 \times \frac 14,\frac 34, \frac18, \frac58$ &$2,4,8$&$2,4,8$& (64,32)  \\
   \hline
1&  -8&$2 \times \frac 34,2\times \frac 18,2\times \frac 58$ &$2,8, 8$&$2,8, 8$& (16,5) \\
   \hline
1&  -8&$4 \times \frac 12,2 \times \frac 34,\frac 1{12},\frac 5{12}$ &$2,4,12$&$2,4,12$& (24,5)\\
   \hline
\end{tabular}
\end{table}

\begin{table}
\caption{Product-quotient surfaces not of general type with $p_g=q=0$, $K^2 \geq -8$, $\gamma=2$ and singularities of multiplicity at most $44$}
\label{gamma2}
\renewcommand{\arraystretch}{1,3}
 \begin{tabular}{|c|c|c|c|c|c|c|c|}
\hline
$\gamma$&$K^2_S$&Sing X&$t_1$&$t_2$&$G$\\
\hline\hline
2&  -3&$\frac 12,\frac 13, \frac 23, \frac 14, 2 \times \frac 18$ &$2,3,8$&$4,6,8$& (192,181)  \\
   \hline
2&  -4&$5 \times \frac 12,\frac 14, 2 \times \frac 18$ &$2,4,8$&$4,4,8$& (64,8)  \\
   \hline
2&  -5&$3 \times \frac 12,3\times \frac 13, 3 \times \frac 16$ &$3,6,6$&$3,6,6$& (18,5) \\
   \hline
2&  -6&$8 \times \frac 13,2\times \frac 16$ &$3,3,6$&$3,3,6$& (36,11) \\
   \hline
2&  -6&$7 \times \frac 12,\frac 14, 2 \times \frac 18$ &$2,4,8$&$2,4,8$& (128,75) \\
   \hline
2&  -6&$\frac 14, \frac 34, 2 \times \frac 18, 2 \times \frac 38$ &$2,4,8$&$2,2,8,8$& (32,9) \\
   \hline
2&  -6&$\frac 12, 2 \times \frac 13, 2 \times \frac 23, \frac 14, 2 \times \frac 18$ &$2,3,8$&$3,4,8$&(192,181) \\
   \hline
2&  -6&$3 \times \frac 12, 2 \times \frac 14, \frac 34, 2 \times \frac 18$ &$2,4,8$&$2,2,4,8$& (32,9) \\
   \hline
2&  -6&$7 \times \frac 12, 2 \times \frac 15, \frac 1{10}$ &$2,5,10$&$2,5,10$&(50,3) \\
   \hline
2&  -6&$4 \times \frac 12, 2 \times \frac 14, \frac 1{12}, \frac 5{12}$ &$2,4,12$&$2,2,4,12$& (24,5) \\
   \hline
2&  -6&$2 \times \frac 12, 4 \times \frac 13, \frac 1{12}, \frac 7{12}$ &$2,3,12$&$3,6,12$&(72,42) \\
   \hline
2&  -7&$2 \times \frac 13,2\times \frac 23, 5 \times \frac 15$ &$3,3,5$&$3,3,5$&(75,2) \\
   \hline
2&  -7&$5 \times \frac 12,3\times \frac 13, 3 \times \frac 16$ &$2,6,6$&$2,6,6$& (36,12)  \\
   \hline
2&  -7&$5 \times \frac 12,\frac 13, \frac 23, \frac 14, 2 \times \frac 18$ &$2,3,8$&$2,6,8$&192,181) \\
   \hline
2&  -7&$2 \times \frac 12,\frac 13, \frac 23, 2 \times \frac 14, \frac 1{12}, \frac 5{12}$ &$2,3,12$&$4,12,12$& (48,33) \\
   \hline
2&  -8&$4 \times \frac 12,2\times \frac 23, 4 \times \frac 16$ &$2,4,6$&$2,2,6,6$& (24,8) \\
   \hline
2&  -8&$4 \times \frac 12,2\times \frac 23, 4 \times \frac 16$ &$2,6,6$&$2,2,6,6$& (12,5)  \\
   \hline
2&  -8&$6 \times \frac 12,2 \times \frac 18, 2 \times \frac 38$ &$2,4,8$&$2,2,8,8$& (16,8) \\
   \hline
2&  -8&$6 \times \frac 12,2 \times \frac 18, 2 \times \frac 38$ &$2,3,8$&$2,2,8,8$& (48,29) \\
   \hline
2&  -8&$9 \times \frac 12, \frac 14, 2 \times \frac 18$ &$2,4,8$&$2,4,8$& (64,8) \\
   \hline
2&  -8&$2 \times \frac 12, 2 \times \frac 14, 2 \times \frac 18, 2 \times \frac 58$ &$2,8,8$&$2,8,8$& (32,5) \\
   \hline
2&  -8&$2 \times \frac 12, \frac 14, \frac 34, 2 \times \frac 18, 2 \times \frac 38$ &$2,8,8$&$4,8,8$& (16,5) \\
   \hline
2&  -8&$6 \times \frac 12, \frac 45, 2 \times \frac 1{10}$ &$2,4,10$&$2,2,2,10$& (40,8) \\
   \hline
2&  -8&$4 \times \frac 12, 4 \times \frac 13, \frac 1{12}, \frac 7{12}$ &$2,6,12$&$2,6,12$& (24,10) \\
   \hline
2&  -8&$2 \times \frac 12, 4 \times \frac 13, 2 \times \frac 23, \frac 14, \frac 1{12}$ &$2,3,12$&$2,3,12$& (192,194) \\
   \hline
2&  -8&$4 \times \frac 13, \frac 14, \frac 34, \frac 1{12}, \frac 7{12}$ &$2,3,12$&$3,4,12$&(72,42) \\
   \hline
      2&  -8&$2 \times \frac 3{20}, 2 \times \frac 14,  6 \times \frac 12$ &$2,4,20$&$2,4,20$& (40,5) \\
      \hline
\end{tabular}
\end{table}

\begin{table}
\caption{Product-quotient surfaces not of general type with $p_g=q=0$, $K^2 \leq -9$, $\gamma=2$ and singularities of multiplicity at most $44$}
\label{gamma2b}
\renewcommand{\arraystretch}{1,3}
 \begin{tabular}{|c|c|c|c|c|c|c|c|}
\hline
$\gamma$&$K^2_S$&Sing X&$t_1$&$t_2$&$G$\\
\hline\hline
2&  -9&$5 \times \frac 15,4\times \frac 25$ &$5,5,5$&$5,5,5$& (5,1) \\
   \hline
2&  -9&$\frac 13, \frac 23, 2 \times \frac 17,4\times \frac 27$ &$2,3,7$&$3,7,7$& (168,42) \\
   \hline
2&  -9&$4 \times \frac 12,5 \times \frac 14, \frac 18, \frac 58$ &$2,4,8$&$2,4,8$& (32,11) \\
   \hline
2&  -9&$\frac 13, \frac 23, 2 \times \frac 14, 2 \times \frac 18, 2 \times \frac 58$ &$2,3,8$&$3,8,8$& (96,64) \\
   \hline
2&  -9&$7 \times \frac 12, \frac 13, \frac 23, 2 \times \frac 15, \frac 1{10}$ &$2,3,10$&$2,3,10$& (150,5) \\
   \hline
2&  -9&$4 \times \frac 12, \frac 13, \frac 23, 2 \times \frac 14, \frac 1{12}, \frac 5{12}$ &$2,12,12$&$3,4,12$& (12,2) \\
   \hline
2&  -9&$2 \times \frac 13, 2 \times \frac 23, \frac 1{13}, 2 \times \frac 3{13}$ &$3,3,13$&$3,3,13$& (39,1) \\
   \hline
    2&  -9&$\frac 1{13}, 2 \times \frac 3{13}, 2 \times \frac 13, 2 \times \frac 23$ &$3,3,13$&$3,3,13$& (39,1) \\
    \hline
2&  -10&$2 \times \frac 13,4\times \frac 23, 4 \times \frac 16$ &$3,6,6$&$3,6,6$& (6,2)  \\
   \hline
2&  -10&$4 \times \frac 12, 2 \times \frac 17,4\times \frac 27$ &$2,3,7$&$2,7,7$& (168,42) \\
   \hline
2&  -10&$2 \times \frac 12,2 \times \frac 13,2 \times \frac 23,2 \times \frac 18, 2 \times \frac 38$ &$2,3,8$&$3,8,8$& (48,29) \\
   \hline
2&  -10&$7 \times \frac 12,2 \times \frac 14,\frac 34,2 \times \frac 18$ &$2,4,8$&$2,4,8$& (32,9) \\
   \hline
2&  -10&$8 \times \frac 12,2 \times \frac 18, 2 \times \frac 38$ &$2,8,8$&$2,8,8$&(8,1) \\
   \hline
2&  -10&$4 \times \frac 12, 2 \times \frac 14, 2 \times \frac 18, 2 \times \frac 58$ &$2,8,8$&$2,8,8$& (16,6) \\
   \hline
2&  -10&$4 \times \frac 12, \frac 14, \frac 34, 2 \times \frac 18, 2 \times \frac 38$ &$2,8,8$&$4,8,8$& (8,1) \\
   \hline
2&  -10&$5 \times \frac 12, 2 \times \frac 13, 2 \times \frac 23, \frac 14, 2 \times \frac 18$ &$2,3,8$&$2,3,8$& (192,181) \\
   \hline
2&  -10&$2 \times \frac 12, 2 \times \frac 15, 3 \times \frac 25, \frac 1{10}, \frac 3{10}$ &$2,5,10$&$5,10,10$& (10,2) \\
   \hline
2&  -10&$6 \times \frac 12, 4 \times \frac 13, \frac 1{12}, \frac 7{12}$ &$2,3,12$&$2,3,12$& (72,42) \\
   \hline
2&  -10&$5 \times \frac 12, 3 \times \frac 23, \frac 14, \frac 16, \frac 1{12}$ &$2,3,12$&$2,6,12$& (48,33) \\
   \hline
2&  -10&$8 \times \frac 12, 2 \times \frac 14, \frac 1{12}, \frac 5{12}$ &$2,4,12$&$2,4,12$& (24,5) \\
   \hline
2&  -11&$4 \times \frac 12, \frac 15, 2 \times \frac 45, 2 \times \frac 1{10}$ &$2,10,10$&$2,10,10$& (20,5) \\
   \hline
2&  -12&$6 \times \frac 12, \frac 14, \frac 34, \frac45, 2 \times \frac 1{10}$ &$2,4,10$&$2,4,10$& (40,8) \\
   \hline
2&  -12&$2 \times \frac 12, 4 \times \frac 25, \frac45, 2 \times \frac 1{10}$ &$2,5,10$&$5,10,10$& (10,2) \\
   \hline
   2&  -12&$7 \times \frac 12, \frac 1{16}, \frac 14,  \frac 7{16}, \frac 34$ &$2,4,16$&$2,4,16$& (32,19) \\
   \hline
2&  -14&$2 \times \frac 56, 2 \times \frac 1{12}, 2 \times \frac 7{12}$ &$2,12,12$&$2,12,12$& (24,9) \\
   \hline
2&  -14&$5 \times \frac 12, 2 \times \frac 34, \frac 56, 2 \times \frac 1{12}$ &$2,4,12$&$2,4,12$& (48,14) \\
   \hline
   2&  -14&$6 \times \frac 12, 2 \times \frac 34, \frac 9{20},  \frac 1{20}$ &$2,4,20$&$2,4,20$& (40,5) \\
   \hline
\end{tabular}
\end{table}

\begin{small}
\begin{table}
\caption{Product-quotient surfaces not of general type with $p_g=q=0$, $\gamma=3$ and singularities of multiplicity 
at most $13$}
\label{gamma3}
\renewcommand{\arraystretch}{1,3}
 \begin{tabular}{|c|c|c|c|c|c|c|c|}
\hline
$\gamma$&$K^2_S$&Sing X&$t_1$&$t_2$&$G$\\
\hline\hline
3&  -9&$7 \times \frac 13, 4 \times \frac 16$ &$3,6,6$&$3,6,6$& (18,3) \\
   \hline
3&  -10&$12 \times \frac 14$ &$4,4,4$&$4,4,4$& (16,2) \\
   \hline
3&  -10&$3 \times \frac 23, 6 \times \frac 16$ &$3,3,6$&$3,3,6$& (108,22) \\
   \hline
3&  -12&$2 \times \frac 12,3 \times \frac 23,6 \times \frac 16$ &$2,6,6,6$&$3,6,6$&  (6,2)  \\
   \hline
3&  -12&$7 \times \frac 12, 2 \times \frac 13, 5 \times \frac 16$ &$2,6,6$&$2,6,6$& (36,12)  \\
   \hline
 3&  -15&$5 \times \frac 17, 4 \times \frac 37, 5 \times \frac 16$ &$7,7,7$&$7,7,7$& (7,1)  \\
   \hline
 3&  -13&$3 \times \frac 17, 6 \times \frac 27$ &$2,3,7$&$7,7,7$& (168,42)  \\
   \hline
 3&  -13&$3 \times \frac 17, 6 \times \frac 27$ &$7,7,7$&$7,7,7$& (7,1)  \\
   \hline
 3&  -12&$3 \times \frac 18, 3 \times \frac 38,5 \times \frac 12$ &$2,3,8$&$2,8,8,8$& (48,29)  \\
   \hline
3&  -10&$2 \times \frac 18, 5 \times \frac 14,5 \times \frac 12$ &$2,4,8$&$2,4,8$& (128,75)  \\
   \hline
3&  -16&$4 \times \frac 18, 2 \times \frac 34,8 \times \frac 12$ &$2,8,8$&$2,8,8$& (16,5)  \\
   \hline
3&  -12&$2 \times \frac 18, 5 \times \frac 14,7 \times \frac 12$ &$2,4,8$&$2,4,8$& (64,8)  \\
   \hline
3&  -6&$2 \times \frac 18, 5 \times \frac 14,12$ &$2,8,8$&$2,8,8$& (64,6)  \\
   \hline
3&  -10&$4 \times \frac 18, 2 \times \frac 34,2 \times \frac 12$ &$2,4,8,8$&$2,8,8$& (16,5)  \\
   \hline
3&  -12&$4 \times \frac 18, 2 \times \frac 34,4 \times \frac 12$ &$2,4,8$&$2,2,8,8$& (32,9)  \\
   \hline
3&  -12&$4 \times \frac 18, 2 \times \frac 34,4 \times \frac 12$ &$2,8,8$&$2,2,8,8$& (16,5)  \\
   \hline
3&  -8&$4 \times \frac 18, 2 \times \frac 34$ &$2,4,8$&$2,2,2,8,8$& (32,9)  \\
   \hline
3&  -16&$4 \times \frac 18, 2 \times \frac 14,4 \times  \frac34$ &$4,8,8$&$4,8,8$& (8,1)  \\
   \hline
3&  -12&$2 \times \frac 18, 6 \times \frac 14,2 \times \frac 58$ &$4,8,8$&$4,8,8$& (8,1)  \\
   \hline
3&  -13&$2 \times \frac 19, 2 \times \frac 29, 5 \times \frac 13,2 \times \frac 23$ &$3,9,9$&$3,9,9$& (9,1)  \\
   \hline
3&  -9&$3 \times \frac 19, 2 \times \frac 13, 3 \times \frac 23$ &$3,3,9$&$3,3,9$& (81,9)  \\
   \hline
3&  -13&$2 \times \frac 19, 3 \times \frac 29, 3 \times \frac 13, \frac 49$ &$3,9,9$&$9,9,9$& (9,1)  \\
   \hline
3&  -6&$3 \times \frac 19, 2 \times \frac 23, \frac 13$ &$3,9,9$&$3,9,9$& (27,2)  \\
   \hline
3&  -12&$\frac 1{12}, 5 \times \frac 14, 4 \times \frac 13, 2 \times \frac 23$ &$3,4,12$&$3,4,12$& (12,2)  \\
   \hline
   3&  -11&$\frac 1{12}, 2 \times \frac 16, \frac 14, \frac 34, 3 \times \frac 13,  \frac 7{12}$ &$3,12,12$&$4,6,12$& (12,2)  \\
   \hline
      3&  -14&$2 \times \frac 1{12}, 2 \times \frac 13,  2 \times \frac 7{12}, 8 \times \frac 23$ &$2,3,12$&$3,12,12$& (72,42)  \\
   \hline
     3&  -12&$2 \times \frac 1{12}, 2 \times \frac 5{12},  6 \times \frac 12$ &$2,4,12$&$2,2,12,12$& (24,5)  \\
   \hline
 3&  -12&$\frac 1{12}, \frac 16, \frac 14,6 \times \frac 13,  5 \times \frac 12$ &$2,3,12$&$2,3,12$& (144,27)  \\
   \hline
      3&  -16&$2 \times \frac 1{12}, 2 \times \frac 14,  9 \times \frac 12,  \frac 56$ &$2,4,12$&$2,4,12$& (48,14)  \\
   \hline
     3&  -16&$2 \times \frac 1{12}, 2 \times \frac 5{12},  10 \times \frac 12$ &$2,12,12$&$2,12,12$& (12,2)  \\
   \hline
 3&  -13&$ \frac 1{12}, 2 \times \frac 16,  3 \times \frac 13,6 \times \frac 12, \frac 7{12}$ &$2,6,12$&$2,6,12$& (24,10)  \\
   \hline
    3&  -13&$ \frac 1{12}, 2 \times \frac 16,  3 \times \frac 13,6 \times \frac 12, \frac 7{12}$ &$2,3,12$&$2,3,12$& (216,92)  \\
   \hline
    3&  -15&$ \frac 1{12}, 2 \times \frac 14,   \frac 13,6 \times \frac 23$ &$2,3,12$&$3,12,12$& (48,33)  \\
   \hline
 3&  -14&$2 \times  \frac 1{12}, 2 \times \frac 14,  2 \times \frac 13, \frac 12,2 \times \frac 23, \frac 56$ &$3,4,12$&$6,12,12$& (12,2)  \\
   \hline
   3&  -16&$2 \times  \frac 1{12}, \frac 16, 2 \times \frac  5{12},  4 \times \frac 12,  \frac 56$ &$2,12,12$&$6,12,12$& (12,2)  \\
   \hline
     3&  -8&$2 \times  \frac 1{12}, 2 \times \frac 14,  \frac  12,  \frac 56$ &$2,4,12$&$2,2,4,12$& (48,14)  \\
   \hline
    3&  -11&$2 \times  \frac 1{12}, 2 \times \frac 14, \frac 12,  \frac  13,  \frac 23,  \frac 56$ &$2,12,12$&$4,6,12$& (24,9)  \\
   \hline
 \end{tabular}
\end{table}
\end{small}

We skip the details of the proof, which is rather long (since the cases are many)
and most of the times straightforward, repeating arguments already used in this paper. Still in a few cases quite some effort is needed to show that the surface is not of general type. Unfortunately, we do not know a systematic way to prove thatcertain product-quotient surfaces cannot be of general type.

\begin{rem}
Note that the surfaces of general type we have found have singular points
of multiplicity much smaller than the bounds in Proposition \ref{nothingmore}, giving some evidence 
to the conjecture that there are no other examples. Still, we can't prove it without proving first 
Conjecture \ref{boundc} at least in the case $p_g=q=0$, finding $\Gamma(0,0)$ explicitly. 
\end{rem}

\section{Upper bounds for $\gamma$ under some additional hypotheses}\label{evidences}

In this section we will give some evidence to Conjecture \ref{boundc}, establishing an upper bound $\Gamma(p_g,q)$ for 
$\gamma$ for product-quotient surfaces of general type under some additional hypotheses. 

Write
$$
K_S=P+N=\sigma^*K_X-\A,
$$
where $P+N$ is the Zariski decomposition of the canonical divisor of the product-quotient surface $S$. 

\begin{rem}
By construction 
$P$,  $\sigma^*K_X$ are nef, $N$, $\A$ are effective; and 
$$
PN = \sigma^* K_X \A = 0.
$$
In particular, $K_S^2 = P^2 +N^2 =  K_X^2 + \A^2$.
\end{rem}
Recall that
 \begin{itemize}
 \item $K_S^2 = 8\chi - 2 \gamma -l$;
 \item $1 \leq P^2 \in \NN$;
 \item $\nu:=-N^2$ is the number of  (-1)-cycles on $S$;
%\item $K_X^2 = 2 \xi =8 \chi +4 \gamma -2 \mu$;
\item $ - \A^2=K_S\A=6\gamma+l-2\mu \geq 0$.
 \end{itemize}

Let $\delta\geq 0$ be a number smaller than
$E \sigma^*K_X$, for each exceptional divisor of the first kind $E$ on $S$.
Then  $$N\sigma^*K_X\geq \nu \delta,$$
which implies $K_S \A  \geq N \A= N (\sigma^*K_X-K_S) \geq \nu (1+\delta)$
and therefore
$$
\nu \leq \frac{6\gamma+l-2\mu}{1+\delta}
$$
so
\begin{equation}\label{1+delta} 
%N\sigma^*K_X\geq \delta N^2 \Rightarrow
1 \leq P^2 =K^2_S-N^2\leq 8\chi +\left(\frac{6}{1+\delta}- 2\right) \gamma -\frac{\delta}{1+\delta} l -\frac{2\mu}{1+\delta}
\end{equation}

\begin{rem} 
Since $K_X$ is nef, we can set $\delta=0$; substituting $\delta=0$ in  (\ref{1+delta}) we obtain $1\leq 2\xi$, 
giving a further proof of Corollary \ref{4chi+gamma-mu>=1/2}.

Writing $E'$ for the unique irreducible component with self intersection (-1) of an exceptional divisor of the first kind $E$, 
we note that, since $\sigma(E')$ is a curve, $K_X$ is ample and $IK_X$ Cartier,  $E\sigma^*K_X\geq E'\sigma^*K_X\geq \frac{1}{I}$. 
So equation (\ref{1+delta}) holds for $\delta =\frac{1}{I}$. 
\end{rem}

\begin{rem}
Assume we could have $\delta =5$, then (\ref{1+delta}) would give a proof of 
Conjecture \ref{boundc}, since
\begin{equation*}
N\sigma^*K_X\geq -5N^2 \Rightarrow
\gamma \leq 8\chi - 1 -\frac{5}{6} l -\frac{\mu}{3}\leq 8\chi - 1
\end{equation*}

\end{rem}
 
 Unfortunately, in general we cannot hope in such a big number $\delta$, since in the fake Godeaux case described in 
\cite{bp} we have:
$$\nu=2, \ E_1\sigma^*K_X=1, \ E_2\sigma^*K_X=11/7,$$ 
hence
 $$N\sigma^*K_X=-9/7 N^2.$$
 Still we can use this inequality
 to prove Conjecture \ref{boundc} if we assume instead that the $\Theta_i$ are not too small. In fact, we have the following result.

\begin{prop} If $\Theta_1$ and $\Theta_2$ are at least $\frac52$, then $\gamma < 8\chi-1$.
\end{prop}

\begin{proof}
If $E'$ is not contained in one of the fibres, then
$$\sigma^*K_X E \geq \sigma^*K_X E'=\frac{K_{C_1 \times C_2}\pi^*\sigma(E' )}{|G|}\geq2(\alpha_1 + \alpha_2) \frac{\Theta_1\Theta_2}{\xi} \geq \Theta_1 + \Theta_2.
$$
Else $\sigma(E')$ is the central component of a singular fibre $F_1^{(i)}$ with multiplicity $m_i$,  then 

$$\sigma^*K_X E \geq \sigma^*K_X E'=\frac{K_{C_1 \times C_2}m_i C_2}{|G|}\geq\frac{\Theta_1\Theta_2}{\xi} 2m_i\alpha_1\geq m_i \Theta_1\geq 2 \Theta_1.
$$

We conclude $E \sigma^*K_X \geq 5$ and therefore \ref{1+delta} holds for $\delta= 2 \frac 52 =5$.

\end{proof}

In a similar way we get an upper  bound for $\gamma$ substituting  $\frac52$ with any number strictly bigger that $\frac32$; still the bounds get worse the closer we approach to $\frac32$.

The assumptions on $\Theta_i$ can be replaced by the hypothesis that $E (F_1+F_2)$ is big enough. For example, we can show the following.
\begin{prop}
Assume that for every exceptional divisor of the first kind $E$, $E(F_1+F_2)\geq 85$. 
Then $\gamma \leq 508\chi - 63$.
\end{prop}

\begin{proof} Arguing as in the previous proposition
$$\sigma^*K_X E =\frac{K_{C_1 \times C_2}\pi^*\sigma(E') }{|G|}\geq\frac{\Theta_1\Theta_2}{\xi} 170\alpha_{min}\geq 85\Theta_{min}\geq 2 +\frac{1}{42}.$$ 

Then (\ref{1+delta} ) holds for $\delta=2 +\frac{1}{42}$ and

\begin{equation*}
1 \leq 8\chi -\frac{2}{127}\gamma -\frac{85}{127}l -\frac{84}{127}\mu
\Rightarrow
\gamma \leq 508\chi -\frac{85}{2}l -42\mu- \frac{127}{2}
\end{equation*}
\end{proof}

\section{The dual surface of a product-quotient surface}\label{duality}

In this section we assume furthermore that $S$ is {\em regular}, i.e., $q(S) = 0$.
 
Suppose that $S$ is given by a pair of generating vectors: $(a_1, \ldots, a_s)$, $(b_1, \ldots , b_t)$ of $G$.

\begin{defin}
 The {\em dual surface $S'$ of $S$} is the product-quotient surface given by the pair of generating vectors: $(a_1, \ldots, a_s)$, $(b_t^{-1}, \ldots , b_1^{-1})$.

Siilarly we will denote by $X'$ the quotient model of $S'$.
\end{defin}

\begin{rem}\label{dualsings}
It is easy to see that $\frac{1}{n}(1,q) \in \mathfrak{B}(X) \ \iff 
\frac{1}{n}(1,n-q) \in \mathfrak{B}(X')$.
\end{rem}

The numbers of $S'$ are then immediatelly computed by those of $S$ as follows.

\begin{prop}
 Let $S$ be a regular product-quotient surface, and denote by $S'$ its dual surface.
Set $\gamma:=\gamma(X)$, $\mu:=\mu(X)$, $l:=l(X)$,
 $\gamma':=\gamma(X')$, $\mu':=\mu(X')$, $l':=l(X')$. Then:
\begin{enumerate}
 \item $\gamma = - \gamma'$;
\item $\mu = \mu '$, 
\item $\xi = \xi'$;
\item $p_g(S') = p_g(S) + \gamma$. 
\end{enumerate}
\end{prop}
\begin{proof}
Remark \ref{dualsings} describes the basket of the singularities of $X'$ in terms of the basket of $X$.

Directly by the definition, and proposition \ref{gammasym}
$$\gamma = -\gamma', \ \ \mu = \mu', \ \ l = l'$$

Then
$$
\chi(S') = \frac{(g_1-1)(g_2-1)}{|G|} + \frac 14 (\mu-2\gamma') = \frac{(g_1-1)(g_2-1)}{|G|} + \frac 14 (\mu+2\gamma)=
$$
$$
= \chi(S) + \gamma.
$$

In particular, since we assumed $q(S) =0$, then 
$$
p_g(S') = p_g(S) + \gamma.
$$
\end {proof}
Note that this gives  an independent proof of Proposition \ref{gamma+pg}. Moreover
\begin{cor}
The dual surface of a product-quotient surface with $p_g=0$ has 
maximal Picard number.
\end{cor}

For the index of $S$ resp. $S'$ we have:
\begin{prop}
$$
\tau(S) := \frac 13 (K_S^2 -2e(S)) = - \frac 13 B(\mathfrak{B}(X)) = -2\gamma -l,
$$

$$
\tau(S') := \frac 13 (K_{S'}^2 -2e(S')) = 2\gamma -l'.
$$
\end{prop}

\begin{proof}
Note that 

\begin{multline*}
e(S) = 2 + 2p_g(S) + h^{1,1}(S) = 2 +2p_g(S) + 2 + 2(\gamma +p_g(S)) +l = \\
= 4(1+p_g(S)) + 2 \gamma + l = 4 \chi(S) + 2 \gamma +l.
\end{multline*}

Therefore
 \begin{multline*}
  \tau(S):=\frac 13 (K_S^2 -2 e(S)) = \frac 13(8\chi(S) -2 \gamma - l- 2 e(S))= \\
= \frac 13(8\chi(S) -2 \gamma -l- 8\chi(S) - 4 \gamma - 2l) = \\
= \frac 13(- 6 \gamma  -3l) =   - 2 \gamma - l.
 \end{multline*}

From the previous calculation it follows that $\tau(S') = -2 \gamma' -l'$. Using $\gamma' = - \gamma$ we get the second equation.
\end{proof}

\begin{rem}
 Let $\bar{S}$ be the minimal model of $S$, then $\tau(S) +(-N^2) = \tau(\bar{S})$. Moreover, by Serrano, we know that for the minimal model of a product-quotient surface, it holds: $\tau(\bar{S}) <0$.

In particular, we get that $l' > 2\gamma$.
\end{rem}

It follows immediately from the above:
$$
\frac 13(B(\mathfrak{B}) + B(\mathfrak{B}')) = l+l' = -(\tau(S) +\tau(S')).
$$

And it is also easy to see that
$$
\frac 13 B(\mathfrak{B}) = l+l'+\tau(S'),
$$
$$
\frac 13 B(\mathfrak{B'}) = l+l'+\tau(S).
$$

\begin{rem}
 Observe that when we go from $S$ to the dual surface $S'$, we consider on $C_1$ the same action of $G$ as for $S$, whereas for $C_2$ we replace the action $y \mapsto g(y)$ by $y \mapsto \overline{g(\overline{y})}$.

Similarly we can replace $y \mapsto g(y)$ by $y \mapsto g \alpha(y)$ for any (holomorphic) automorphism $\alpha$ of 
$C_2$, getting many new surfaces from this construction (depending on the representation theory of $G$).
\end{rem}

%%%%%%%%%%%%%%%%%%%%%%%%%%%%%%%%%%%%%%%%%%%%%%%%%%%%%%

\bigskip
\noindent
{\bf Authors' Adresses:}

\noindent
{\em Ingrid Bauer}\\
 Mathematisches Institut der Universit\"at Bayreuth\\
Universit\"atsstr. 30;\\
D-95447 Bayreuth, Germany\\
\noindent
{\em Roberto  Pignatelli} \\
Dipartimento di Matematica, Universit\`a di Trento;\\
Via Sommarive 14;
I-38123 Trento (TN), Italy\\

\bigskip
\noindent

\end{document}